\documentclass[reqno]{amsart}
\usepackage{graphicx} 
\newcommand{\kk} {\Bbbk} 
\newcommand{\fd} {\mathfrak{d}} 
\newcommand{\fD} {\mathfrak{D}} 

\usepackage{answers}
\usepackage{tikz}
\usepackage{tikz-cd}
\usepackage{graphicx}
\usepackage{mathrsfs}
\usepackage{enumitem}
\usepackage{multicol}
\usepackage{mathrsfs}
\usepackage{float}
\usepackage[
  hmarginratio={1:1},     
  vmarginratio={1:1},     
  textwidth=420pt,        
  heightrounded,          
]{geometry}
\usepackage{amsmath,amsthm,amssymb}
\usepackage{hyperref}
\usepackage{pdfsync}
\usepackage{helvet}
\usepackage{bm}
\usepackage{todonotes}
\usepackage{mathtools}
\usepackage{stmaryrd}
\usepackage{relsize}
\usepackage[intoc]{nomencl}

\usetikzlibrary{decorations.markings}
\usetikzlibrary{shapes.misc}
\tikzset{->-/.style={decoration={
  markings,
  mark=at position .5 with {\arrow{>}}},postaction={decorate}}}
  \tikzset{cross/.style={cross out, draw=black, minimum size=2*(#1-\pgflinewidth), inner sep=0pt, outer sep=0pt},
cross/.default={1pt}}

\theoremstyle{plain}
\newtheorem{thm}{Theorem}[section]
\newtheorem{cor}[thm]{Corollary}

\newtheorem*{notation*}{Notation}

\newtheorem{defn}[thm]{Definition}
\newtheorem{definition}[thm]{Definition}
\newtheorem{lemma}[thm]{Lemma}

\theoremstyle{plain}
\newtheorem{conj}[thm]{Conjecture}
\theoremstyle{definition}

\newtheorem{rmk}[thm]{Remark}

\newcommand{\Pic}{\text{Pic}}

\newcommand{\sD}{\mathscr{D}}

\def \Spec {\operatorname{Spec}}

\def \var {\vartheta}

\graphicspath{{images/}}

\title{Intrinsic mirror symmetry for orbifold cubic surfaces and the $A_1$ spherical DAHA}

\author[P.\,Bousseau]{Pierrick Bousseau}
\address{The Mathematical Institute, University of Oxford, Oxford, OX2 6GG, UK}
\email{pierrick.bousseau@maths.ox.ac.uk}

\author[S.\,Chattopadhyay]{Sayan Chattopadhyay}
\address{Department of Mathematics,
University of Georgia, Athens, GA 30605, USA}
\email{Sayan.Chattopadhyay@uga.edu}

\date{}
\numberwithin{equation}{section}
\begin{document}
	\begin{abstract}
		We determine the intrinsic mirror of an orbifold cubic surface with three nodes and interpret the resulting geometry through mirror symmetry between the $PGL(2,\mathbb{C})$ and $SL(2,\mathbb{C})$ character varieties of the once-punctured torus. We further construct a quantization of the mirror using higher genus Gromov--Witten theory and show that it recovers the $A_1$ spherical double affine Hecke algebra.
	\end{abstract}
\maketitle
\setcounter{tocdepth}{1}
\tableofcontents
\section{Introduction}

\subsection{Intrinsic mirror symmetry for orbifold cubic surfaces}

Let $\kk$ be an algebraically closed field of characteristic zero, and let $(Y,D)$ be a maximal log Calabi--Yau pair over $\kk$, namely a pair consisting of a smooth projective variety $Y$ over $\kk$ and a reduced anticanonical normal crossings divisor $D \subset Y$ which contains a zero-dimensional stratum. The intrinsic mirror construction of Gross--Hacking--Keel in dimension two \cite{GHK1} and of Gross--Siebert in higher dimension \cite{GScanonical, gross2019intrinsic} associates to $(Y,D)$ a commutative $\kk$-algebra $R_{(Y,D)}$, called the \emph{intrinsic mirror algebra}, and defined in terms of log Gromov--Witten invariants counting rational curves in $Y$ satisfying prescribed tangency conditions along $D$. The affine scheme
\[
\Spec R_{(Y,D)}
\]
is then a candidate for the total space of the mirror family of $(Y,D)$ in the various specific formulations of mirror symmetry. An explicit description of the algebra $R_{(Y,D)}$ requires a priori the computation of log Gromov--Witten invariants, a task that is generally highly nontrivial, and only completed in a few examples \cite{arguz_quartic, arguz_equations, HDTV,  B_explicit, GHK1, GHKScubic}.

The intrinsic mirror symmetry construction is expected to generalize to pairs $(Y,D)$ in which $Y$ is an orbifold, that is, a smooth Deligne--Mumford stack. Such a generalization has not yet been carried out, primarily because the requisite theory of orbifold logarithmic Gromov--Witten theory has not yet been fully developed. In dimension two, however, when the orbifold points are contained in $Y \setminus D$, the orbifold and logarithmic structures do not interact. In this setting, orbifold logarithmic Gromov--Witten theory can be defined in a straightforward manner; see \cite[\S 5.5]{GPS}. Consequently, the intrinsic mirror construction of \cite{GHK1} admits a natural generalization to this setting-- see \S \ref{Section 4} for details.

In this paper, we determine the intrinsic mirror algebra for a particular orbifold cubic surface. More precisely, we consider $Y$ to be 
a projective cubic surface with three nodes ($A_1$-singularities), regarded as an orbifold, and $D \subset Y$ to be a triangle of lines avoiding the nodes. 
As reviewed in \S \ref{Section 2}, the surface $Y$ can be realized as follows. Let $\overline{D}_1,\overline{D}_2,\overline{D}_3$ be three lines forming a triangle in $\mathbb{P}^2$, and choose non-collinear points $p_i\in \overline{D}_i$, for $i=1,2,3$, away from the vertices of the triangle. Then $Y$ is obtained by blowing up $\mathbb{P}^2$ along the three length-two zero-dimensional subschemes obtained by fattening each $p_i$ in the direction of $\overline{D}_i$, and $D=D_1+D_2+D_3$ is given by the strict transform of $\overline{D}_1+\overline{D}_2+\overline{D}_3$.
The exceptional curves $E_1$, $E_2$, $E_3$ are each lines passing through one node and have self-intersection $-\frac{1}{2}$. Finally, denote by $H$ the class of the strict transform of a line in $\mathbb{P}^2$, so that $D_i=H-2E_i$ for all $1 \leq i \leq 3$. The classes $H, E_1, E_2, E_3$ form a basis of the group of numerical equivalence curve classes on $Y$.

\begin{thm} \label{thm_main_intro_1}
   Let $(Y,D)$ be an orbifold cubic surface over $\kk$, where $Y$ is a projective cubic surface with three nodes and $D \subset Y$ is a triangle of lines avoiding the nodes. Let $\operatorname{NE(Y)}$ be the monoid of effective curve classes on $Y$ and let $\kk[\operatorname{NE(Y)}]$ be the corresponding monoid $\kk$-algebra. 
   The intrinsic mirror algebra 
   $R_{(Y,D)}$ of $(Y,D)$ is the quotient of the $\kk[\operatorname{NE(Y)}]$-algebra $\kk[\operatorname{NE(Y)}][\var_{v_1}, \var_{v_2}, \var_{v_3}]$ by the relation
    \[\var_{v_1} \var_{v_2} \var_{v_3} = t^{D_1} \var_{v_1}^2 +t^{D_2} \var_{v_2}^2 +t^{D_3} \var_{v_3}^2 + t^{H} + t^{2 D-H} - 2t^{D}\,,\]
    where $t^{C}$ denotes the monomial in
    $\kk[\operatorname{NE(Y)}]$ corresponding to a curve class $C \in \operatorname{NE(Y)}$.
\end{thm}

The intrinsic mirror family given by Theorem \ref{thm_main_intro_1} admits a particularly interesting one-parameter subfamily described as follows. 
The class $\sD = 2H-2E_1-2E_2-2E_3 \in NE(Y)$ satisfies $\sD \cdot D_i=0$ for all $1 \leq i \leq 3$, $\sD \cdot D=0$ and $\sD \cdot H=2$. Thus, restricting
the intrinsic mirror family to the one-dimensional torus 
\[ \mathbb{G}_m =\Spec \, \kk[\mathbb{Z}] \hookrightarrow \Spec\, \kk[NE(Y)] \] defined by the monoid homomorphism
$NE(Y) \rightarrow \mathbb{Z}$, 
$C \mapsto  \sD \cdot C$, we obtain 
\begin{equation} \label{eq_cubic}
\var_{v_1} \var_{v_2} \var_{v_3} = \var_{v_1}^2 +\var_{v_2}^2 + \var_{v_3}^2 + t^2 + t^{-2} - 2\,. \end{equation}

For smooth maximal log Calabi--Yau surfaces, a deformation quantization $\widehat{R}_{(Y,D)}$ of the intrinsic mirror algebra $R_{(Y,D)}$ was constructed in \cite{bousseau_mirror} using higher-genus log Gromov--Witten invariants with a top $\lambda$-class insertion. The resulting quantum mirror algebra is noncommutative and defined over $\kk_q :=\kk[q^{\pm \frac{1}{2}}]$, where the parameter $q$ is related to the deformation quantization parameter $\hbar$ by $q=e^{i\hbar}$.
The intrinsic mirror algebra $R_{(Y,D)}$ is recovered in the classical limit $\hbar\to 0$, or equivalently $q\to 1$. This construction extends naturally to orbifold log Calabi--Yau surfaces, and we determine the quantum mirror algebra associated with a cubic surface with three nodes:

\begin{thm} \label{thm_main_intro_2}
Let $(Y,D)$ be an orbifold cubic surface over $\kk$, where $Y$ is a projective cubic surface with three nodes, and $D \subset Y$ is a triangle of lines avoiding the nodes. 
The quantum intrinsic mirror algebra
$\widehat{R}_{(Y,D)}$
of $(Y,D)$ 
is the associative $\kk_q[\operatorname{NE}(Y)]$-algebra generated by three elements 
$\widehat{\vartheta}_{v_1}, \widehat{\vartheta}_{v_2}, \widehat{\vartheta}_{v_3}$ that satisfy the relations, 
	\begin{align*}
		q^{-\frac{1}{2}}\widehat{\vartheta}_{v_{1}}\widehat{\vartheta}_{v_{2}} - q^{\frac{1}{2}}\widehat{\vartheta}_{v_{2}}\widehat{\vartheta}_{v_{1}} & = \left(q^{-1} - q \right) \widehat{\vartheta}_{v_{3}}\,, \\
		q^{-\frac{1}{2}}\widehat{\vartheta}_{v_{2}}\widehat{\vartheta}_{v_{3}} - q^{\frac{1}{2}}\widehat{\vartheta}_{v_{3}}\widehat{\vartheta}_{v_{2}} & = \left(q^{-1} - q \right) \widehat{\vartheta}_{v_{1}}\,, \\
		q^{-\frac{1}{2}}\widehat{\vartheta}_{v_{3}}\widehat{\vartheta}_{v_{1}} - q^{\frac{1}{2}}\widehat{\vartheta}_{v_{1}}\widehat{\vartheta}_{v_{3}} & = \left(q^{-1} - q \right) \widehat{\vartheta}_{v_{2}}\,, \\
	\end{align*}
and
	\begin{align*}q^{-\frac{1}{2}}\widehat{\vartheta}_{v_1}\widehat{\vartheta}_{v_2}
    \widehat{\vartheta}_{v_3} = q^{-1} t^{D_1} \widehat{\vartheta}_{v_1}^2 + q t^{D_2} \widehat{\vartheta}_{v_2}^2 + q^{-1} t^{D_3} \widehat{\vartheta}_{v_3}^2 & + t^{H} + t^{2D-H} -(q+q^{-1}) t^{D} \,.\end{align*}
\end{thm}

For a smooth cubic surface with boundary given by a triangle of lines, the intrinsic mirror algebra was determined by Gross--Hacking--Keel--Siebert in~\cite{GHKScubic}, while its quantum counterpart was computed by the first author in~\cite{bousseau_skein}. Our proof follows the same general strategy. The main new feature is the presence of orbifold points, which requires adapting the computation of the relevant log Gromov--Witten invariants to the orbifold setting.

\subsection{Mirror symmetry for character varieties}

A major motivation for studying mirror symmetry in the case of orbifold cubic surfaces with three nodes is that it provides an explicit example of mirror symmetry for character varieties.

For integers $g,n \geq 0$, denote by $\mathbb{S}_{g,n}$ the compact, connected, oriented topological surface of genus $g$ with $n$ punctures. For any  connected reductive algebraic group $G$ over $\kk$, the $G$-character variety of $\mathbb{S}_{g,n}$ is the affine variety over $\kk$ defined by the affine GIT quotient   \[\operatorname{Ch}_{g,n}(G) = \operatorname{Hom}\left(\pi_{1}(\mathbb{S}_{g,n}), G\right)/\!/G\,,\]
where $G$ acts by conjugation on the affine variety $\operatorname{Hom}\left(\pi_{1}(\mathbb{S}_{g,n}), G\right)$ of representations of the fundamental group 
$\pi_1(\mathbb{S}_{g,n})$ into $G$.
Character varieties are objects that lie in the intersection of low-dimensional topology, algebraic geometry, and representation theory -- see \cite{lub, sik} for references and background. 
The relative character varieties 
$\operatorname{Ch}_{g,n}^{c}(G)$ are the fibers $(\pi_{g,n}^G)^{-1}(c)$ of the map \[\pi_{g,n}^G : \operatorname{Ch}_{g,n}(G) \to \left(G/\!/G\right)^n\]
that sends a representation to the conjugacy classes associated to the small oriented loops around the punctures. 

Mirror symmetry is expected to relate character varieties associated with a pair of Langlands dual groups $G$ and $G^\vee$. For $\kk=\mathbb{C}$,  these character varieties are homeomorphic via non-abelian Hodge theory to coarse moduli spaces of semistable Higgs bundles,
whose Hitchin fibrations form a pair of dual SYZ fibrations 
\cite{donagi_pantev, hausel}. 
It is natural to expect a version of mirror symmetry for character varieties formulated in terms of the intrinsic mirror construction. 
The relative character varieties are conjectured to admit log Calabi--Yau compactifications $(Y,D)$ -- see \cite{arguz2026logcalabiyaucompactificationssl2mathbbc, kim2026compactification, MR4157427} for results in this direction. To formulate the following conjecture, we assume that an appropriate version of the intrinsic mirror construction is available for the pair \((Y,D)\). This is certainly the case when \((Y,D)\) is a normal crossings pair. In general, however, \((Y,D)\) may be singular, in which case the intrinsic mirror construction does not apply directly.

\begin{conj} \label{conj_intro}
Let \(G\) be a connected reductive algebraic group over \(\kk\), and let \(G^\vee\) denote its Langlands dual. For a general point \(c\in (G/\!/G)^n\), let
$(Y,D)$
be a log Calabi--Yau compactification of the relative character variety
$\operatorname{Ch}^{c}_{g,n}(G)$.
After a suitable base change, the intrinsic mirror family associated with \((Y,D)\) is isomorphic to the family
$\pi_{g,n}^{G^\vee}\colon
\operatorname{Ch}_{g,n}(G^\vee)
\longrightarrow
(G^\vee/\!/G^\vee)^n
$ of relative $G^\vee$-character varieties.
\end{conj}

Conjecture \ref{conj_intro} follows for $G=PGL(2)$, $G^\vee=SL(2)$ and $(g,n) = (0,4)$ from \cite{bousseau_skein, GHKScubic}. 
A topological version of mirror symmetry in this case is also proved in \cite{mauri}, and a version of homological mirror symmetry is established in \cite{beimler2026mirror}.

We now explain how Theorem \ref{thm_main_intro_1} implies Conjecture \ref{conj_intro} for
$G=PGL(2)$, $G^\vee=SL(2)$ and $(g,n) = (1,1)$, that is, when the surface $\mathbb{S}_{g,n}$ is a once-punctured torus. The trace of $2 \times 2$ matrices induces an isomorphism $SL(2)/\!/SL(2) \simeq \mathbb{A}^1$. As reviewed in \cite{goldman2011affinecubicsurfacesrelative}, for every $\tau \in \mathbb{A}^1$, the relative character variety $\mathrm{Ch}^{ \tau}_{1,1}(SL(2))$ is given by the affine cubic surface in $\mathbb{A}^3$ with equation
\begin{equation} \label{eq_character_variety}
xyz=x^2+y^2+z^2-\tau -2 \,.\end{equation}
Indeed, writing $\pi_1(\Bbb{S}_{1,1}) = \langle \alpha,\beta, \mathfrak{r} \ \vert \ \alpha\beta\alpha^{-1}\beta^{-1} = \mathfrak{r} \rangle$, equation \eqref{eq_character_variety} is obtained by setting  
\[x = \operatorname{Tr}(\rho(\alpha)), \ \ y = \operatorname{Tr}(\rho(\beta)), \ \ z = \operatorname{Tr}(\rho(\alpha\beta)), \ \  \tau = \operatorname{Tr}(\rho(\mathfrak{r}))\]
for $\rho : \pi_{1}(\Bbb{S}_{1,1}) \to SL(2)$. In particular, $\mathrm{Ch}^{\tau}_{1,1}(SL(2))$ is smooth if and only if $\tau \neq \pm 2$. For $\tau=2$, it is the Cayley cubic, which has four nodes, whereas for $\tau=-2$, it is the Markoff cubic, which has a unique node.
On the other hand, since $PGL(2)=SL(2)/\{ \pm \mathrm{Id}\}$, we have 
\[\mathrm{Ch}^{\tau}_{1,1}(PGL(2))=\mathrm{Ch}^{\tau}_{1,1}(SL(2))/H^1(\mathbb{S}_{1,1}, \mathbb{Z}/2\mathbb{Z}) = \mathrm{Ch}^{\tau}_{1,1}(SL(2))/(\mathbb{Z}/2 \mathbb{Z})^2 \,,  \]
where the non-trivial elements of 
$(\mathbb{Z}/2 \mathbb{Z})^2$ act on the cubic surface \eqref{eq_character_variety} by 
\begin{align*}
    (x,y,z) & \mapsto (-x,-y,z) \,,\\
    (x,y,z) & \mapsto (-x,y,-z)  \,,\\ 
    (x,y,z) & \mapsto (x,-y,-z) \,.
\end{align*}
The invariant subalgebra is generated by $u = x^2$, $v=y^2$, $w = z^2$, and so the relative character variety 
$\mathrm{Ch}^{\tau}_{1,1}(PGL(2))$ is the affine cubic surface in $\mathbb{A}^3$ with equation 
\begin{equation}\label{1}uvw = (u + v + w -
\tau -2)^2 \,.\end{equation}
In particular, $\mathrm{Ch}^{\tau}_{1,1}(PGL(2))$ has three nodes for $\tau \neq \pm 2$, four nodes for $\tau=2$, and a unique $D_4$ singularity for $\tau=-2$. 
For general $\tau$, $\mathrm{Ch}^{\tau}_{1,1}(PGL(2))=Y \setminus D$, where $Y$ is a projective cubic surface with three nodes, and $D \subset Y$ a triangle of lines avoiding the nodes. Therefore, comparing \eqref{eq_cubic} with \eqref{eq_character_variety}, we obtain the following corollary of Theorem \ref{thm_main_intro_1}:

\begin{cor}
After the change of variables 
 $\tau=-t^2-t^{-2}$, the 
family \eqref{eq_character_variety} of relative $SL(2)$-character varieties 
$\mathrm{Ch}^{\tau}_{1,1}(SL(2))$ coincides with the base change \eqref{eq_cubic} of the intrinsic mirror family to $(Y, D)$, where $Y$ is a projective cubic surface with three nodes and $D \subset Y$ a triangle of lines avoiding the nodes. In particular, Conjecture \ref{conj_intro} holds for $G=PGL(2)$, $G^\vee=SL(2)$ and $(g,n) = (1,1)$. 
\end{cor}

\subsection{Quantum mirror symmetry and spherical double affine Hecke algebra}
We explain how Theorem \ref{thm_main_intro_2} recovers the $A_1$ spherical double affine Hecke algebra.
Double affine Hecke algebras (DAHAs) were first introduced in \cite{Cherednik1} and are labeled by simple Lie algebras. We give a brief review of the DAHA of type $A_1$, that is, associated to the Lie algebra
$\mathfrak{sl}_2$, following the conventions of \cite[\S 2.2]{BranesDAHA}.

Let $\kk_{q,t}$ be the localization of $\kk[q^{\pm \frac{1}{2}}, t^{\pm 1}]$ at the multiplicative system generated by elements of the form $q^\ell t - q^{-\ell} t^{-1}$ with $\ell \in \mathbb{Z}_{\geq 0}$.
The $A_1$ DAHA is the $\kk_{q,t}$-algebra 
\begin{align*}
	\Bbb{H}_{q,t} = \kk_{q,t}[X^{\pm 1}, Y^{\pm 1}, T^{\pm 1}]/\left \langle \begin{array}{cc}TXT=X^{-1}~,
		& Y^{-1}X^{-1}YXT^2=q^{-1}~,\\
		TY^{-1}T=Y~,& (T-t^{1})(T+t^{-1})=0\end{array}\right \rangle \,.
\end{align*}
The $A_1$ spherical DAHA is the subalgebra $\mathbb{SH}_{q,t}$ of $\Bbb{H}_{q,t}$
defined by 
\[ \mathbb{SH}_{q,t} = e \Bbb{H}_{q,t}e \,,\]
where  $e \in \Bbb{H}_{q,t}$ is the idempotent 
$e = \dfrac{1 + tT}{1+t^2}$, which satisfies the relations $eT = Te = te$,  $eT^{-1} = T^{-1}e = t^{-1}e$. 
As reviewed in \cite[\S 2.2, Eq. (2.50)]{BranesDAHA}, $\mathbb{SH}_{q,t}$ is isomorphic to the associative
$\kk_{q,t}$-algebra generated by 
\begin{align*}
	&x = e(X + X^{-1})e  \,,\\
	&y = e(Y + Y^{-1})e \,,\\
	&z = e\left(q^{-\frac{1}{2}}Y^{-1}X + q^{\frac{1}{2}}X^{-1}Y\right)e \,, 
\end{align*}
which satisfy the relations,
\begin{align*}
	 q^{-\frac{1}{2}}xy - q^{\frac{1}{2}}yx  &= (q^{-1}-q)z  \,,\\
	 q^{-\frac{1}{2}}yz - q^{\frac{1}{2}}zy &= (q^{-1} - q)x \,,\\
	q^{-\frac{1}{2}}zx - q^{\frac{1}{2}}xz  &= (q^{-1} - q)y \,, \\
  q^{-\frac{1}{2}}xyz  = 	q^{-1}x^2 + qy^{2} +& q^{-1}z^{2}  -  q^{-1}t^{2} - q t^{-2} - q - q^{-1} \,.
\end{align*}
Therefore, comparison with Theorem \ref{thm_main_intro_2} establishes the following:

\begin{cor} \label{cor_intro}
Let $(Y,D)$ be an orbifold cubic surface over $\kk$, where $Y$ is a projective cubic surface with three nodes, and $D \subset Y$ is a triangle of lines avoiding the nodes. 
Then, the base change of the quantum intrinsic mirror algebra
$\widehat{R}_{(Y,D)}$
of $(Y,D)$ by the map \[ \mathbb{G}_m =\Spec \, \kk[\mathbb{Z}] =\Spec \kk[t^\pm]\hookrightarrow \Spec\, \kk[NE(Y)] \] defined by 
$NE(Y)\rightarrow \mathbb{Z}$, 
$[C] \mapsto  [\sD] \cdot [C]$, is isomorphic to the $A_1$ spherical DAHA 
\[ \mathbb{SH}_{q, i q^{\frac{1}{2}}t}\,.\]
\end{cor}

Finally, we note that the $A_1$ spherical double affine Hecke algebra arises in theoretical physics as the algebra of line operators in the four-dimensional $\mathcal{N}=2^\star$ gauge theory with gauge group $SU(2)$. Equivalently, this theory can be realized as the Class $S$ theory of type $A_1$ associated with $\mathbb{S}_{1,1}$; see, for instance, \cite{BranesDAHA, gupta2026spherical, TV} and the references therein. As explained in \cite[\S 1.3]{bousseau_skein} for the case of $\mathbb{S}_{0,4}$, the description of the spherical DAHA in terms of the intrinsic mirror construction given by Corollary \ref{cor_intro} agrees with the expected description of the algebra of line operators in terms of BPS states.

\subsection{Outline of the paper}
In \S\ref{Section 2}, we introduce orbifold cubic surfaces and discuss their basic properties. 
In \S\ref{Section 4}, we define and compute the canonical scattering diagram for cubic surfaces with three nodes, which we use in \S\ref{Section 5} to determine the intrinsic mirror algebra. Finally, we turn to the quantum setting: in \S\ref{Section 6}, we compute the quantum canonical scattering diagram, and in \S\ref{Section 7}, we use this computation to determine the quantum intrinsic mirror algebra.

\subsection{Acknowledgments} 
The first author acknowledges the support of a Sloan Research Fellowship from the Alfred P. Sloan Foundation. For the purpose of open access, the authors have applied a CC BY public copyright licence to any author accepted manuscript arising from this submission.

\section{Orbifold cubic surfaces}
\label{Section 2}

\subsection{Cubic surfaces with three nodes}
In this paper, $Y$ denotes a projective cubic surface in $\mathbb{P}^3$ with three nodes (ordinary double points, or $A_1$-singularities), and $D \subset Y$ is a triangle of lines disjoint from the nodes. Since nodes are quotient singularities, étale locally of the form $\mathbb{A}^2/(\mathbb{Z}/2\mathbb{Z})$, the surface $Y$ admits a natural interpretation as a smooth Deligne--Mumford stack. For brevity, we refer to the pair $(Y,D)$ as an \emph{orbifold cubic surface with three nodes}.

All such surfaces admit the following construction. Start with the projective plane $\mathbb{P}^2$, equipped with a triangle of lines
\[
\overline{D}=\overline{D}_1+\overline{D}_2+\overline{D}_3.
\]
Choose points $p_i\in\overline{D}_i$, for $1\leq i\leq 3$, such that $p_1,p_2,p_3$ are not collinear. Let
$
Y'\longrightarrow \mathbb{P}^2$
be the blow-up at $p_1,p_2,p_3$, with exceptional curves $E_1',E_2',E_3'$, and let
\[
D'=D_1'+D_2'+D_3'
\]
be the strict transform of $\overline{D}$.

Next, let
$
Y''\longrightarrow Y'$
be the blow-up at the three points $E_i'\cap D_i'$, for $1\leq i\leq 3$, and let
\[
D''=D_1''+D_2''+D_3''
\]
be the strict transform of $D'$. Denote by $E_{ii}''$ the exceptional curves of $Y''\to Y'$; these are $(-1)$-curves. We denote by $E_i''$ the strict transforms of the curves $E_i'$; these are $(-2)$-curves.

Finally, let $Y$ be the surface obtained from $Y''$ by contracting the three $(-2)$-curves $E_i''$, and let $D$ be the image of $D''$. The resulting pair $(Y,D)$ is an orbifold cubic surface with three nodes $x_1,x_2,x_3$. Conversely, every orbifold cubic surface with three nodes arises in this way, since each step of the construction can be reversed. In particular, the morphism
$Y''\longrightarrow Y$
is the minimal resolution of $Y$; see Figure~\ref{Figure 1}. We denote by $E_i$ the images of the curves $E_{ii}''$.

\begin{figure}[h!]
    \centering
    \includegraphics[width=7cm]{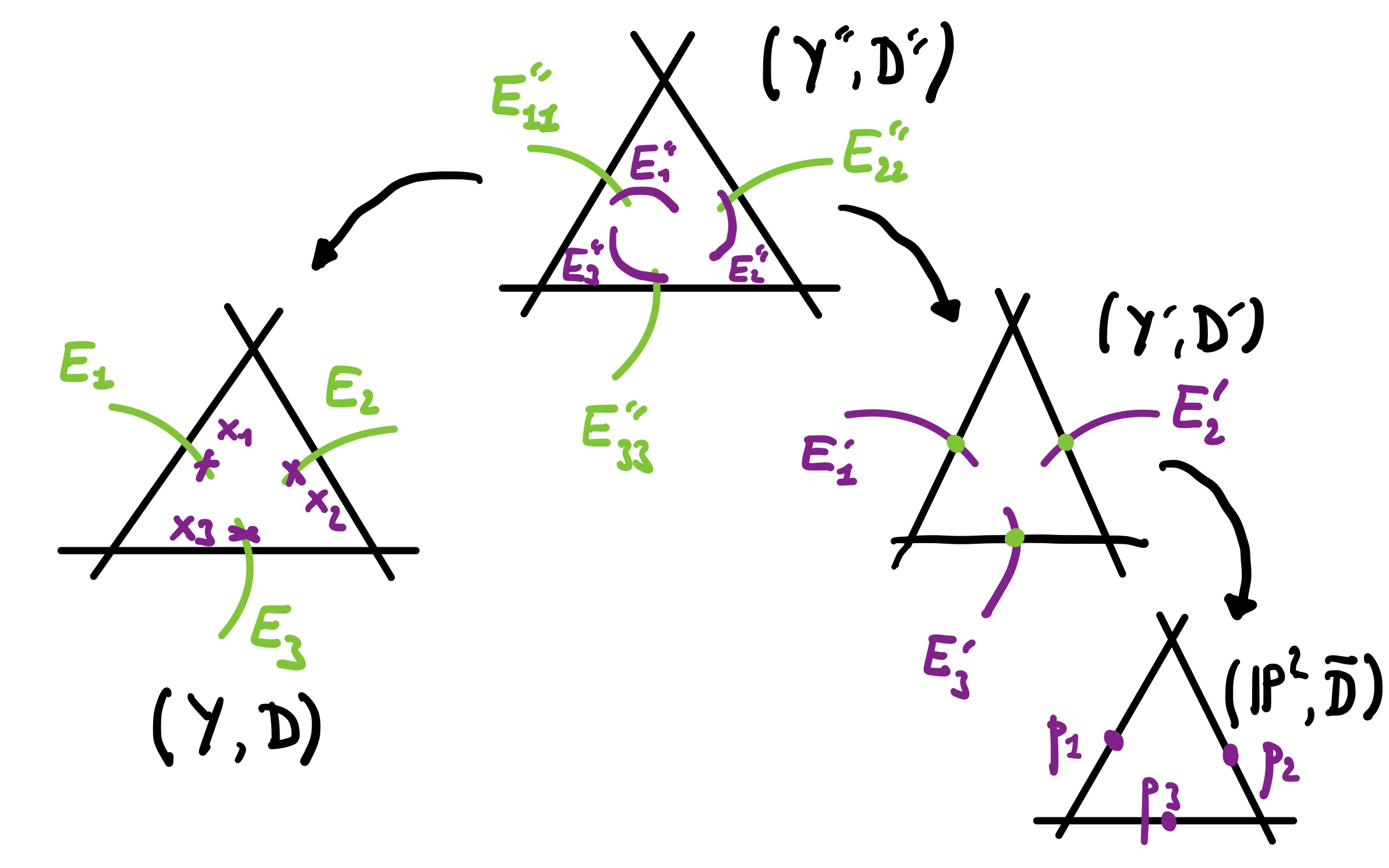}
         \caption{Cubic surfaces}
              \label{Figure 1}
     \end{figure}

Up to isomorphism, orbifold cubic surfaces are determined by the positions of the points $p_i$ on the lines $\overline{D}_i$, modulo the action of the torus $(\mathbb{G}_m)^2$ on $\mathbb{P}^2$ preserving the triangle $\overline{D}$. Consequently, orbifold cubic surfaces with three nodes form a one-parameter family.
An explicit one-parameter family of cubic surfaces containing all orbifold cubic surfaces with three nodes is given by
\[
xyz = x^2 + y^2 + z^2 + 2t^2(x+y+z) + (t^4-4t^2-4).
\]
For $t\neq 0,\pm2$, this defines an orbifold cubic surface with three nodes. The exceptional values correspond to degenerations: at $t=0$, the cubic surface has four nodes, whereas at $t=\pm 2$, it has a singularity of type $D_4$.

In this paper, we study curves on $Y$. We begin by recalling the configuration of lines on $Y$. By \cite[\S 2.14]{dolgachev_cubic}, an orbifold cubic surface with three nodes contains exactly $12$ lines in $\mathbb{P}^3$. These lines can be divided into three types according to the number of nodes they contain.
First, exactly three lines are disjoint from the nodes; these are precisely the components $D_1,D_2,D_3$ of $D$. Next, there are six lines containing exactly one node. More precisely, through each node $x_i$ there are exactly two such lines, whose classes are
\[
L_1^{v_i}=E_i,
\qquad
L_2^{v_i}=2H-2E_j-2E_k-E_i,
\]
where $\{j,k\}=\{1,2,3\}\setminus\{i\}$ and $H$ denotes the pullback of the class of a line in $\mathbb{P}^2$.
Finally, there are three lines containing two nodes. For each pair of distinct nodes $x_i$ and $x_j$, there is a unique line passing through both, whose class is
\[
H-E_i-E_j.
\]

\subsection{Tropical cubic surfaces}

Finally, we introduce the tropicalization of $(Y,D)$. Tropicalizations of log Calabi--Yau surfaces are introduced in \cite{GHK1}. This construction only depends on the self-intersection numbers of the components of $D$. Therefore, for orbifold cubic surfaces, the tropicalization is identical to the tropicalization of smooth cubic surfaces studied in \cite{GHKScubic} and can be described as follows.

Let $B$ be the quotient of $\mathbb{R}^2$ by the involution
\[
(x,y)\longmapsto (-x,-y).
\]
The standard integral affine structure determined by the lattice $\mathbb{Z}^2\subset\mathbb{R}^2$ descends to an integral affine structure on $B\setminus\{0\}$. This structure does not extend across the origin as a nonsingular integral affine structure, so that $0\in B$ is an integral affine singularity. We denote by
$B(\mathbb{Z})$
the set of integral points of $B$.
We now define a cone complex $\Sigma$ on $B$. It is convenient to realize $B$ as the upper half-plane
\[
\mathbb{H}=\{(x,y)\in\mathbb{R}^2\mid y\geq 0\},
\]
with the boundary identification
$(x,0)\sim(-x,0)$.
Consider the integral vectors
\[
v_1=(1,0),\qquad v_2=(0,1),\qquad v_3=(-1,1),
\]
and let
$
\rho_i:=\mathbb{R}_{\geq 0}v_i$, $i=1,2,3$,
be the corresponding one-dimensional cones. The two-dimensional cones are
$\sigma_{i,i+1}:=\langle v_i,v_{i+1}\rangle$, 
where the indices are taken modulo $3$. Thus,
\[
\Sigma
=
\left\{
\{0\},
\rho_1,\rho_2,\rho_3,
\sigma_{12},\sigma_{23},\sigma_{31}
\right\}
\]
defines a cone complex in $B$.
The pair $(B,\Sigma)$ is the tropicalization of $(Y,D)$. We will refer to $(B,\Sigma)$ as the \emph{tropical cubic surface}.

\section{The canonical scattering diagram}
\label{Section 4}

 In \S \ref{sec_def_canonical}, we define the canonical scattering diagram for orbifold cubic surfaces. The main result of this section is the explicit calculation of this canonical scattering diagram in \S\ref{sec_calculation_canonical}-\ref{Subsection 4.1} 
 by direct evaluation of genus zero orbifold log Gromov--Witten invariants.

\subsection{Definition of the canonical scattering diagram}
\label{sec_def_canonical}

Let $(B, \Sigma)$ be the tropical cubic surface defined in \S\ref{Section 2}. Recall that $\kk[NE(Y)]$ is the monoid algebra of the monoid $NE(Y)$ of effective curve classes on $Y$. An ideal $I \subset \kk[NE(Y)]$ is called co-Artinian if the quotient $\kk[NE(Y)]/I$ is an Artinian $\kk$-algebra.

\begin{defn} A \emph{ray} in $(B, \Sigma)$ is a pair $(\mathfrak{d}, f_{\mathfrak{d}})$, where \begin{itemize}
    \item[(i)] $\mathfrak{d} \subset B$ is a ray $\mathbb{R}_{\geq 0}v_\fd$ generated by some primitive $v_{\fd} \in B(\Bbb{Z}) \setminus \{0\}$.
    \item[(ii)] $f_{\mathfrak{d}} \in \kk[NE(Y)][\![z^{-v_\fd}]\!]$ such that $f_{\mathfrak{d}} =  1 \operatorname{mod}z^{-v_\fd}$, and for every co-Artinian ideal $I \subset \kk[NE(Y)]$, $f_{\mathfrak{d}}\, \mathrm{mod}\, I$ is a finite sum.
\end{itemize}    
\end{defn}
\begin{defn}
    A \emph{scattering diagram} $\mathfrak{D}$ is a collection of rays $(\mathfrak{d}, f_{\mathfrak{d}})$ such that:
   \begin{itemize} 
\item[(i)]    $\Bbb{R}_{\geq 0}v_{\mathfrak{d}} = \Bbb{R}_{\geq 0}v_{\mathfrak{d}'}$ implies $\mathfrak{d}=\mathfrak{d}'$,
\item[(ii)] for every co-Artinian ideal $I \subset \kk[NE(Y)]$, there are only finitely many rays $(\mathfrak{d}, f_{\mathfrak{d}})$ such that $f_{\mathfrak{d}} \neq 1\, \mathrm{mod}\, I$.
    \end{itemize}
\end{defn}

Let $(Y,D)$ be an orbifold cubic surface with three nodes. Fix a curve class $\beta \in \operatorname{NE}(Y)$ and an integral point $v \in B(\Bbb{Z})$. We denote by
$\overline{\mathcal{M}}^{\beta}_{0,v}(Y,D)$
the moduli space of genus-zero orbifold basic stable log maps to $(Y,D)$ with one marked point of contact order $v$ along $D$.
The theory of stable log maps to log smooth log schemes was developed in \cite{AC, GSlog}, while the theory of orbifold stable maps to smooth Deligne--Mumford stacks was established in \cite{orbifold_GW}. A general theory of orbifold stable log maps to log orbifolds has not yet been developed. In the present setting, however, this causes no additional difficulty. Indeed, the orbifold points of $(Y,D)$ are disjoint from $D$. Consequently, the points of the domain curve carrying non-trivial orbifold structure are disjoint from those carrying non-trivial logarithmic structure. The orbifold and logarithmic structures can therefore be treated independently, allowing one to define $\overline{\mathcal{M}}^{\beta}_{0,v}(Y,D)$ and establish the required properties by combining the corresponding theories. This situation is also discussed and used in \cite[\S 5.5]{GPS}.
 Thus, $\overline{\mathcal{M}}^{\beta}_{0,v}(Y, D)$ is a proper Deligne--Mumford stack and carries a zero-dimensional virtual class $[\overline{\mathcal{M}}^{\beta}_{0,v}(Y, D)]^{\mathrm{vir}}$.
 We define genus zero orbifold log Gromov--Witten invariants:
\[N^{\beta}_{0, v} := \operatorname{deg}
[\overline{\mathcal{M}}^{\beta}_{0,v}(Y, D)]^{\mathrm{vir}}  \in \mathbb{Q} \,.\]
We can now define the canonical scattering diagram associated with $(Y,D)$ as in \cite{GHK_moduli}, with the only modification arising from the orbifold setting.

\begin{defn}
Let $(Y,D)$ be an orbifold cubic surface with three nodes. 
    The \emph{canonical scattering diagram} $\mathfrak{D}_{\text{can}}$ of $(Y,D)$ is the collection of rays $(\mathfrak{d}, f_{\mathfrak{d}})$, where $\mathfrak{d} = \Bbb{R}_{\geq 0}v$ with primitive $v \in B(\mathbb{Z})$ and 
    \[f_{\mathfrak{d}} = \operatorname{exp}\left(\sum_{k\geq 1}\sum_{\beta \in NE(Y)} k N_{0, kv}^\beta t^{\beta}z^{-kv}\right) \in \kk[NE(Y)][\![z^{-v}]\!]\,.\]
\end{defn}

\subsection{The ray $(\rho_1, f_{\rho_1})$ in the canonical scattering diagram}
\label{sec_calculation_canonical}

In this section, we determine the function $f_{\rho_1}$ attached to the ray $\rho_1 = \mathbb{R}_{\geq 0}v_1$ in the canonical scattering diagram 
$\mathfrak{D}_{\text{can}}$. We first identify the curve classes which could contribute non-trivially to $f_{\rho_1}$. We will use the following notation: \[ L_1^{v_1}:=E_1,\,\,\,\,\, L_2^{v_1}:=2H-2E_2-2E_3-E_1,\,\,\,\,\, \text{and} \,\,\,\,C^{v_1}:=2H-2E_2-2E_3\,.\]

\begin{lemma} \label{lem_classification}
Let $\beta \in NE(Y)$ be the curve class of a generically injective map $f: C \rightarrow Y$ such that $\beta \cdot D_1=k$  for some $k \geq 1$
and $\beta \cdot D_2 = \beta \cdot D_3=0$, where $C$ is an integral curve. Then, either $k=1$, in which case $\beta=L_1^{v_1}$, $L_2^{v_1}$, or $H-E_2-E_3$, or $k=2$, in which case $\beta=C^{v_1}$.
\end{lemma}

\begin{proof}
We follow the same strategy as in the proof of
\cite[Proposition 2.4]{GHKScubic}. As in \S2, let
$\pi \colon Y'' \longrightarrow Y$
be the minimal resolution of $Y$. Let
$D''=D''_1+D''_2+D''_3$ 
be the strict transform of $D=D_1+D_2+D_3$. Denote by
$E''_1,E''_2,E''_3$ the exceptional divisors of $\pi$, which are
$(-2)$-curves, and by $E''_{11},E''_{22},E''_{33}$ the strict
transforms of $E_1,E_2,E_3$, which are $(-1)$-curves. We use the
basis
\[
H'',E''_1,E''_2,E''_3,E''_{11},E''_{22},E''_{33}
\]
of $\Pic(Y'')$, where $H''$ is the pullback of the class of a line
in $\mathbb P^2$. For $1\leq i\leq 3$, we have
\[
D''_i=H''-E''_i-2E''_{ii}.
\]

Let $\beta''$ be the class of the strict transform of
$f(C)$. Since the exceptional curves $E''_i$ are disjoint
from $D''$, the projection formula gives
\[
\beta''\cdot D''_1=k,
\qquad
\beta''\cdot D''_2=\beta''\cdot D''_3=0.
\]
Write
\[
\beta''
=
aH''
-\sum_{i=1}^3 b_iE''_i
-\sum_{i=1}^3 b_{ii}E''_{ii},
\qquad
a,b_i,b_{ii}\in\mathbb Z.
\]
Since $\beta''$ is represented by an effective integral curve and
$H''$ is nef, we have
$a=\beta''\cdot H''\geq 0$.
The intersection conditions with the boundary give
\begin{align}
a-b_{11}&=k, \label{eq:lemma34-1}\\
a-b_{22}&=0, \label{eq:lemma34-2}\\
a-b_{33}&=0. \label{eq:lemma34-3}
\end{align}
Thus
$b_{11}=a-k$ and
$b_{22}=b_{33}=a$.
Since $\beta''$ is represented by an integral curve, its arithmetic
genus is nonnegative. Moreover, $K_{Y''}=-D''$, and hence the
adjunction formula gives
\begin{align*}
-2
&\leq 2p_a(\beta'')-2 =(K_{Y''}+\beta'')\cdot\beta''\\
&=(\beta''-D'')\cdot\beta''
=
a^2-\sum_{i=1}^3 b_i^2
+2\sum_{i=1}^3b_i b_{ii}
-2\sum_{i=1}^3b_{ii}^2-k\\
&=
a^2-\sum_{i=1}^3(b_i-b_{ii})^2
-\sum_{i=1}^3b_{ii}^2-k.
\end{align*}
Using \eqref{eq:lemma34-1}--\eqref{eq:lemma34-3}, we obtain
\begin{equation}\label{eq:lemma34-adjunction}
\begin{split}
-2\leq {}&
a^2-k
-\bigl((a-k-b_1)^2+b_1^2\bigr)\\
&-\bigl((a-b_2)^2+b_2^2\bigr)
-\bigl((a-b_3)^2+b_3^2\bigr).
\end{split}
\end{equation}

For every $n\in\mathbb Z$,
\[
\min_{b\in\mathbb Z}
\left\{(n-b)^2+b^2\right\}
=
\left\lceil\frac{n^2}{2}\right\rceil.
\]
It follows from \eqref{eq:lemma34-adjunction} that
\begin{equation}\label{eq:lemma34-bound}
-2
\leq
a^2-k
-\left\lceil\frac{(a-k)^2}{2}\right\rceil
-2\left\lceil\frac{a^2}{2}\right\rceil.
\end{equation}
If $a$ is even, the right-hand side of
\eqref{eq:lemma34-bound} is
\[
-k-\left\lceil\frac{(a-k)^2}{2}\right\rceil,
\]
whereas if $a$ is odd, it is
\[
-1-k-\left\lceil\frac{(a-k)^2}{2}\right\rceil.
\]
Since $a\geq0$ and $k\geq1$, the only possibilities are
\[
(a,k)=(0,1),\qquad (1,1),\qquad (2,1),\qquad (2,2).
\]
We consider these four cases separately.

\smallskip
\noindent
\emph{Case 1: $(a,k)=(2,2)$.}
In this case,
\[
b_{11}=0,
\qquad
b_{22}=b_{33}=2,
\]
and \eqref{eq:lemma34-adjunction} becomes
\[
-2
\leq
2-2b_1^2
-\bigl((2-b_2)^2+b_2^2\bigr)
-\bigl((2-b_3)^2+b_3^2\bigr).
\]
Since
\[
(2-b)^2+b^2\geq2
\]
for every $b\in\mathbb Z$, with equality if and only if $b=1$,
we must have
\[
b_1=0,\qquad b_2=b_3=1.
\]
Therefore
\[
\beta''
=
2H''
-E''_2-2E''_{22}
-E''_3-2E''_{33}.
\]
Since
\[
\pi_*H''=H,\qquad
\pi_*E''_i=0,\qquad
\pi_*E''_{ii}=E_i,
\]
we obtain
\[
\beta
=
\pi_*\beta''
=
2H-2E_2-2E_3
=
C^{v_1}.
\]

\smallskip
\noindent
\emph{Case 2: $(a,k)=(1,1)$.}
Here
\[
b_{11}=0,
\qquad
b_{22}=b_{33}=1,
\]
and \eqref{eq:lemma34-adjunction} becomes
\[
-2
\leq
-2b_1^2
-\bigl((1-b_2)^2+b_2^2\bigr)
-\bigl((1-b_3)^2+b_3^2\bigr).
\]
For $b\in\mathbb Z$,
\[
(1-b)^2+b^2\geq1,
\]
with equality if and only if $b\in\{0,1\}$. Hence
\[
b_1=0,
\qquad
b_2,b_3\in\{0,1\}.
\]
Thus the numerical class $\beta''$ is one of
\[
\begin{split}
&H''-E''_{22}-E''_{33},\\
&H''-E''_2-E''_{22}-E''_{33},\\
&H''-E''_3-E''_{22}-E''_{33},\\
&H''-E''_2-E''_{22}-E''_3-E''_{33}.
\end{split}
\]
Since the $E''_i$ are contracted by $\pi$, all four possibilities
have the same pushforward $
\beta
=
H-E_2-E_3$.

\smallskip
\noindent
\emph{Case 3: $(a,k)=(2,1)$.}
Here
\[
b_{11}=1,
\qquad
b_{22}=b_{33}=2,
\]
and \eqref{eq:lemma34-adjunction} becomes
\[
-2
\leq
3
-\bigl((1-b_1)^2+b_1^2\bigr)
-\bigl((2-b_2)^2+b_2^2\bigr)
-\bigl((2-b_3)^2+b_3^2\bigr).
\]
We have
\[
(1-b)^2+b^2\geq1,
\]
with equality if and only if $b\in\{0,1\}$, and
\[
(2-b)^2+b^2\geq2,
\]
with equality if and only if $b=1$. Therefore
\[
b_1\in\{0,1\},
\qquad
b_2=b_3=1.
\]
Consequently, the numerical class $\beta''$ is
\[
\beta''
=
2H''
-b_1E''_1-E''_{11}
-E''_2-2E''_{22}
-E''_3-2E''_{33},
\qquad b_1\in\{0,1\}.
\]
Both possibilities have the same pushforward
$
\beta
=
2H-E_1-2E_2-2E_3
=
L^{v_1}_2$.

\smallskip
\noindent
\emph{Case 4: $(a,k)=(0,1)$.}
In this case,
\[
b_{11}=-1,
\qquad
b_{22}=b_{33}=0,
\]
and \eqref{eq:lemma34-adjunction} becomes
\[
-2
\leq
-1-\bigl((-1-b_1)^2+b_1^2\bigr)
-2b_2^2-2b_3^2.
\]
Since
\[
(-1-b)^2+b^2\geq1,
\]
with equality if and only if $b\in\{0,-1\}$, we obtain
\[
b_2=b_3=0,
\qquad
b_1\in\{0,-1\}.
\]
Thus,
$
\beta''=E''_{11}$ or 
$\beta''=E''_1+E''_{11}$.
Since $E''_1$ is
contracted by $\pi$, in either case, we have
$
\beta
=
\pi_*\beta''
=
E_1
=
L^{v_1}_1$. 

Combining the four cases, if $k=1$, then
$
\beta=L^{v_1}_1$,
$\beta=L^{v_1}_2$, or
$\beta=H-E_2-E_3$,
while if $k=2$, then
$\beta=C^{v_1}$. 
This proves the lemma.
\end{proof}

\begin{lemma} \label{lem_lines}
On $Y$, there exists a unique curve of class $L_1^{v_1}$ and a unique curve of class $L_2^{v_1}$. These two curves are exactly the two lines in $\mathbb{P}^3$ contained in $Y$ that pass through the node $x_1$.
\end{lemma}

\begin{proof}
We use the presentation of the minimal resolution $Y''$ of $Y$ as a
sequence of blow-ups of $\mathbb{P}^2$. The uniqueness of the curve in
the class $L_1^{v_1}=E_1$ follows from the uniqueness of the
corresponding exceptional divisor. Under the blow-down to
$\mathbb{P}^2$, a curve in the class
$L_2^{v_1}=2H-E_1-2E_2-2E_3$
is the strict transform of a conic satisfying the incidence and
tangency conditions imposed by the blow-up centers: it is tangent to
two prescribed lines at prescribed points and passes through a third
prescribed point. These conditions determine the conic uniquely.
\end{proof}

\begin{lemma} \label{lem_bad_line}
On $Y$, there exists a unique curve of class $H-E_2-E_3$. This curve is the unique line in $\mathbb{P}^3$ that passes through the nodes $x_2$ and $x_3$.
\end{lemma}

\begin{proof}
We use the presentation of the minimal resolution $Y''$ of $Y$ as a
sequence of blow-ups of $\mathbb{P}^2$. The uniqueness of the curve in
the class $H-E_2-E_3$ follows from the uniqueness of a line passing through two distinct points in $\mathbb{P}^2$. 
\end{proof}

\begin{lemma} \label{lem_conics} 
Assume that the orbifold cubic surface $Y$ is general in its deformation class.
Then, there are exactly two curves on $Y$ of class $C^{v_1}$ that intersect $D_1$ at a single point: a unique smooth conic and the double of the line of class $H-E_2-E_3$ that passes through the nodes $x_2$ and $x_3$. 
\end{lemma}

\begin{proof}
Since
$C^{v_1}=2H-2E_2-2E_3$, 
under the birational morphism to $\mathbb P^2$, curves of class $C^{v_1}$
correspond to plane conics tangent to $\overline D_2$ at $p_2$ and to
$\overline D_3$ at $p_3$. These conics form a pencil.

Choose homogeneous coordinates $[x:y:z]$ on $\mathbb P^2$ such that
\[
\overline D_1=\{z=0\},\qquad
\overline D_2=\{y=0\},\qquad
\overline D_3=\{x=0\}.
\]
Using the torus action preserving the boundary triangle, we may moreover
assume that
$p_2=[1:0:1]$, $p_3=[0:1:1]$.
Writing a general conic as
\[
Q=ax^2+by^2+cz^2+dxy+eyz+fxz,
\]
the conditions that $Q$ be tangent to $\overline D_2$ at $p_2$ and to
$\overline D_3$ at $p_3$ give
$a=b=c$, $e=f=-2a$.
Thus the pencil is
\[
Q_{a,d}
=
a(x^2+y^2+z^2-2xz-2yz)+dxy,
\qquad [a:d]\in\mathbb P^1.
\]

Restricting to $\overline D_1$, we obtain
\[
Q_{a,d}|_{\overline D_1}=ax^2+dxy+ay^2.
\]
Hence $Q_{a,d}$ intersects $\overline D_1$ at a single point if and only
if this binary quadratic form has a double root, or equivalently
$d^2-4a^2=0$.
There are therefore exactly two such members of the pencil.

For $d=2a$, we have
\[
Q_{a,2a}=a(x+y-z)^2.
\]
Thus this member is twice the line through $p_2$ and $p_3$. Its strict
transform on $Y$ has class $H-E_2-E_3$, and hence, by Lemma~3.6, is the
line through the nodes $x_2$ and $x_3$.

For $d=-2a$, after rescaling $a=1$, we obtain
\[
Q_{1,-2}
=
x^2+y^2+z^2-2xy-2xz-2yz.
\]
This conic is smooth, since the determinant of its associated symmetric
matrix
\[
\begin{pmatrix}
1 & -1 & -1\\
-1 & 1 & -1\\
-1 & -1 & 1
\end{pmatrix}
\]
is $-4$. Thus the second member is a unique smooth conic. For general $Y$, the point to blow up on $\overline{D}_1$ is distinct from the intersection point of this smooth conic with $\overline{D}_1$, and so the result follows.
\end{proof}

\begin{lemma}
	\label{Lemma 4.3} Assume that the orbifold cubic surface $Y$ is general in its deformation class.
    Let $f : C \to Y$ be an orbifold stable log map, defining a point in $\overline{\mathcal{M}}^{\beta}_{0,v}(Y, D)$. Then, the image curve $f(C)$ is irreducible.  
 \end{lemma}
\begin{proof}
We first show that, for a general choice of $Y$, the curves appearing in Lemmas \ref{lem_lines}, \ref{lem_bad_line}, \ref{lem_conics} meet $D_1$ at distinct
points.
As before, we use the presentation of the minimal resolution of $Y$ as a sequence of blow-ups of $\mathbb{P}^2$. 
Let
$\overline D_1,\overline D_2,\overline D_3\subset\mathbb{P}^2$
be the toric boundary. Choose homogeneous coordinates $[x:y:z]$ on
$\mathbb{P}^2$ such that
\[
\overline D_1=\{z=0\},\qquad
\overline D_2=\{y=0\},\qquad
\overline D_3=\{x=0\}.
\]
Using the torus action preserving the toric boundary, we may assume
$
p_2=[1:0:1]$, $ p_3=[0:1:1]$, 
and write
$p_1=[u:v:0]\in\overline D_1$, 
where $uv\neq0$.
We determine on $\mathbb{P}^2$ the points of $\overline D_1$ corresponding
to the points at which the relevant curves on $Y$ meet $D_1$.

The curve of class $L_1^{v_1}=E_1$ is exceptional over $p_1$.
Consequently, its intersection point with $D_1$ lies over
$p_1=[u:v:0]$.
The curve of class
$L_2^{v_1}=2H-E_1-2E_2-2E_3$
is the strict transform of the unique plane conic $\overline L_2$
which passes through $p_1$ and is tangent to $\overline D_2$ at
$p_2$ and to $\overline D_3$ at $p_3$. As in the proof of
Lemma \ref{lem_conics}, the pencil of plane conics satisfying the latter two
tangency conditions is
\[
Q_{a,d}
=
a(x^2+y^2+z^2-2xz-2yz)+dxy.
\]
Since $\overline L_2$ passes through $p_1$, we have
$a(u^2+v^2)+duv=0$.
Its intersection with $\overline D_1$ is determined by
$ax^2+dxy+ay^2=0$.
One root is $[u:v]$. Since the product of the two roots is $1$, the
other is $[v:u]$. The first intersection is removed by the blowup at
$p_1$, while the second remains the intersection with the strict
transform of $\overline D_1$. Hence the curve of class
$L^{v_1}_2$ meets $D_1$ at the point corresponding to
\[
[v:u:0]\in\overline D_1.
\]

The curve of class
$H-E_2-E_3$ 
is the strict transform of the line $\overline M\subset\mathbb{P}^2$ through
$p_2$ and $p_3$. Since
$\overline M=\{x+y-z=0\}$,
we have
$
\overline M\cap\overline D_1=\{[1:-1:0]\}$. 
Thus the corresponding curve on $Y$ meets $D_1$ at the point
corresponding to $[1:-1:0]$.

Finally, the smooth conic of class $C^{v_1}$ from Lemma \ref{lem_conics} is the
strict transform of
\[
\overline Q
=
\{x^2+y^2+z^2-2xy-2xz-2yz=0\}.
\]
Since
$\overline Q|_{\overline D_1}=(x-y)^2$, 
it is tangent to $\overline D_1$ at
$[1:1:0]$.
For $Y$ general enough this point is different from $p_1$, so the
strict transform of $\overline Q$ meets $D_1$ at the corresponding
point of $Y$.

Thus, after identifying $D_1$ with $\overline D_1$ away from the
blow-up center $p_1$, the four relevant contact points are represented
by
\[
[u:v:0],\qquad
[v:u:0],\qquad
[1:-1:0],\qquad
[1:1:0].
\]
They are pairwise distinct provided
\[
u\neq v,\qquad u\neq -v.
\]
These are proper closed conditions on the choice of
$p_1\in\overline D_1$. Hence, for $Y$ general enough, any two
distinct curves occurring in Lemmas \ref{lem_lines}, \ref{lem_bad_line}, \ref{lem_conics} meet $D_1$ at
distinct points.

Now let $f\colon C\to Y$ be as in the statement. By the arguments
of \cite[Lemmas~4.9 and~4.10]{bousseau_skein}, formulated for smooth cubic surfaces but which also apply to orbifold cubic surfaces since the tropicalizations are identical, together with Lemma \ref{lem_classification}, every non-contracted
irreducible component of $C$ maps onto one of the curves
described in Lemmas \ref{lem_lines}, \ref{lem_bad_line}, \ref{lem_conics}. Moreover, since $f$ has a single point
of contact with the boundary, the images of all non-contracted
components meet $D_1$ at the same point.
By the pairwise distinctness proved above, these images cannot have
distinct supports. Thus all non-contracted components of $C$ have
the same image in $Y$. Therefore, $f(C)$ is a single
irreducible curve.

\end{proof}

Next, we calculate the orbifold log Gromov--Witten invariants associated to each of the cases identified above.

\begin{lemma} \label{lem_line}
We have 
\[ \exp \left( \sum_{k \geq 1} k N_{0, kv_1}^{k L_1^{v_1}} t^{k L_1^{v_1}} z^{-kv_1}\right) = 1 + t^{2 L_1^{v_1}} z^{-2v_1} \]
and 
\[ \exp \left( \sum_{k \geq 1} k N_{0, kv_1}^{k L_2^{v_1}} t^{k L_2^{v_1}} z^{-kv_1}\right) = 1 + t^{2 L_2^{v_1}} z^{-2v_1}\,. \]
\end{lemma}

\begin{proof}
By Lemmas \ref{lem_classification} and \ref{lem_lines}, every orbifold stable log map contributing to $N_{0, kv_1}^{k L_1^{v_1}}$ factors through the unique line of class $L_1^{v_1}$. This line contains exactly one $\mathbb{Z}/2\mathbb{Z}$ orbifold point and intersects $D$ transversely at a single point. Therefore, the calculation of $N_{0, kv_1}^{k L_1^{v_1}}$ reduces to a calculation in relative orbifold Gromov-Witten theory of the stacky $\mathbb{P}^1$ with a $\mathbb{Z}/2\mathbb{Z}$ orbifold point, and so \cite[Proposition 5.7]{GPS} implies that $N_{0, kv_1}^{k L_1^{v_1}}=0$ if $k$ is odd, and
$N_{0, kv_1}^{k L_1^{v_1}} = \frac{(-1)^{l-1}}{2l^2}$ if $k=2l$ is even. Therefore, we obtain
\[ \exp \left( \sum_{k \geq 1} k N_{0, kv_1}^{k L_1^{v_1}} t^{k L_1^{v_1}} z^{-kv_1}\right) = \exp \left(\sum_{l \geq 1} 2l \frac{(-1)^{l-1}}{2l^2} t^{2lL_1^{v_1}} z^{-2lv_1} \right) = 1 + t^{2 L_1^{v_1}} z^{-2v_1} \,. \]
The same argument applies for $L_1^{v_1}$ replaced by $L_2^{v_1}$.
\end{proof}

\begin{lemma} \label{lem_conic}
We have \[ \exp \left( \sum_{l \geq 1} (2l) N_{0, 2lv_1}^{l C^{v_1}} t^{l C^{v_1}} z^{-2lv_1}\right) = (1 - t^{C^{v_1}} z^{-2v_1})^{-2} \,. \]
Moreover, we have $N_{0, kv_1}^{k(H-E_2-E_3)}=0$ for every $k$ odd. 
\end{lemma}

\begin{proof}
By deformation invariance of orbifold log Gromov--Witten invariants, one can assume that $Y$ is general enough.
By Lemmas \ref{lem_classification}, \ref{lem_bad_line}, and \ref{lem_conics}, every orbifold stable log map contributing to $N_{0, kv_1}^{k(H-E_2-E_3)}$ factors either through the unique smooth conic of class $C^{v_1}$ tangent to $D_1$, in which case $k$ is even since $C^{v_1}=2(H-E_2-E_3)$, or factors through the unique line passing through the nodes $x_2$ and $x_3$. We claim that the orbifold stable log maps factoring through this line do not contribute to the orbifold log Gromov--Witten invariant. Indeed, since this line intersects $D_1$ transversely and has self-intersection $(H-E_2-E_3)^2=0$, the virtual class for genus zero orbifold stable log maps factoring through this line is a multiple of the Euler class of the bundle with fibers $H^1(\mathcal{O})=0$,  and so vanishes.

Therefore, one can assume that $k=2l$ is even, and restrict to orbifold stable log maps
that factor through the smooth conic of class $C^{v_1}$.
This conic does not contain any orbifold point. Therefore, by \cite[Proposition 6.1]{GPS} applied with $w=2$, we have $N_{0, 2lv_1}^{l C^{v_1}} = \frac{1}{l^2}$ for all $l \geq 1$. Therefore, we obtain
\[ \exp \left( \sum_{l \geq 1} (2l) N_{0, 2lv_1}^{l C^{v_1}} t^{l C^{v_1}} z^{-2lv_1}\right) =
\exp \left( \sum_{l \geq 1} (2l)\frac{1}{l^2} t^{l C^{v_1}} z^{-2lv_1}\right) = 
(1 - t^{C^{v_1}} z^{-2v_1})^{-2} \,. \]
\end{proof}

\begin{lemma}
	\label{Lemma 4.1}
    The ray $(\rho_1, f_{\rho_1})$ in the canonical scattering diagram $\mathfrak{D}_{\text{can}}$ is given by 
    \[f_{\rho_{1}} = 
    \dfrac{\left(1+t^{2L_1^{v_1}}z^{-2v_1}\right)\left(1+t^{2L_2^{v_1}}z^{-2v_1}\right)}{\left(1-t^{C^{v_1}}z^{-2v_1}\right)^2} \,,\] 
    where $L_1^{v_1} = E_1, L_2^{v_1} =2H- 2E_2-2E_3- E_1$ and $C^{v_1} = 2H - 2E_2 - 2E_3$.
\end{lemma}

\begin{proof}
By deformation invariance of orbifold log Gromov--Witten invariants, one can assume that $Y$ is general enough.
By Lemmas \ref{lem_classification}, \ref{Lemma 4.3}, together with Lemmas \ref{lem_lines},\ref{lem_bad_line},\ref{lem_conics}, the only curve classes contributing to $f_{\rho_1}$ are multiples of $L_1^{v_1}$, $L_2^{v_1}$, and $H-E_2-E_3$. The contributions of these curve classes are determined by 
Lemmas \ref{lem_line} and \ref{lem_conic}.
\end{proof}

\subsection{General rays in the canonical scattering diagram} \label{Subsection 4.1} 
Following \cite{GHKScubic} in the case of a smooth cubic surface, we explain how to deduce the form of general rays of the canonical scattering diagram from the particular ray $(\rho_1, f_{\rho_1})$ using a $\operatorname{PSL}(2,\Bbb{Z})$ symmetry. 
First note that the natural action of $SL(2,\Bbb{Z})$ by linear transformations on $\Bbb{R}^2$  induces an action of $PSL(2,\Bbb{Z})$ on the tropical cubic surface $B=\mathbb{R}^2/\{ \pm 1 \}$, preserving the set $B(\mathbb{Z})$ of integral points.

On the other hand, there is an action of $PSL(2,\mathbb{Z})$ on the group $A_1(Y)$ of curve classes, described as follows.
Recall that $SL(2,\Bbb{Z})$ is generated by 
\[ S = \begin{pmatrix}
    0&1 \\-1 & 1
\end{pmatrix}, \ \ T =  \begin{pmatrix}
    1&1 \\0&1
\end{pmatrix}\,.\]
Moreover, $A_1(Y)$ is generated by the classes $E_1, E_2, E_3, H$. 
The action of $S$ is defined by 
\[S_*(H) = H, \ \ S_*(E_i) = E_{i+1} \ \ i = 1,2,3 \,(\operatorname{mod} 3)\,.\]
Denote 
\[\Gamma = \{(v,\beta )\in B(\mathbb{Z}) \times \operatorname{NE(Y)} \vert \ N_{0, v} ^{\beta} \neq 0\}\,.\]

\begin{lemma}
	\label{Lemma 4.4}
    The action induced by $S_*$ on $B(\mathbb{Z}) \times A_1(Y)$ induces an action on $\Gamma$, preserving $\Gamma$. Furthermore,
    \[N_{0,S(v)}^{S_*(\beta)} = N_{0,v}^{\beta}\]
\end{lemma}
\begin{proof}
    See \cite[Lemma 4.15]{bousseau_skein}
\end{proof}

The action of $T$ on $A_1(Y)$ is defined by
\[T_*(E_1) = E_1, \ T_*(E_3) = E_2, \ T_*(E_2) = H - E_3, \ T_*(H) = 2H - 2E_2 \,. \]
As in \cite{GHKScubic}, this transformation is induced by the following log birational transformation of $(Y,D)$: 
blow up the intersection of $D_1 \cap D_3$ and contract the $(-1)-$curve $D_2$ to obtain a new orbifold cubic surface $(\widetilde{Y}, \widetilde{D})$ with $\widetilde{D} = \widetilde{D}_1 + \widetilde{D}_2 + \widetilde{D}_3$, where $\widetilde{D}_1$ is the strict transform of $D_1$, $\widetilde{D}_2$ is the strict transform of $D_3$ and $\widetilde{D}_3$ is the exceptional divisor of the blowup at $D_1 \cap D_3$. Note that $\widetilde{Y}$ still has three nodes since the nodes are contained in the complement of $D$. 

\begin{lemma}
	\label{Lemma 4.5}
    The action induced by $T_*$ on $A_1(Y)$ induces an action on $\Gamma$, preserving $\Gamma$. Furthermore,
    \[N_{0, T(v)}^{T_*(\beta)} = N_{0,v}^{\beta}\]
\end{lemma}
\begin{proof}
Since the birational transformation inducing $T$ is a logarithmic birational modification, invariance of logarithmic Gromov--Witten invariants under such modifications \cite{Abwise} gives
$N^{T_\ast(\beta)}_{0,T(v)} = N^\beta_{0,v}$.
The remaining assertions follow as in \cite[Lemma 4.16]{bousseau_skein}.
\end{proof}

\begin{lemma}
    For every primitive $v \in B(\Bbb{Z})$, there exists $M \in PSL(2,\Bbb{Z})$, such that, denoting $L_1^{v}=M_\star(L_1^{v_1})$, 
    $L_2^{v}=M_\star(L_2^{v_1})$, $C^{v}=M_\star(C^{v_1})$, the ray $(\fd, f_\fd)$ of the canonical scattering diagram $\fD_{\text{can}}$ with $\fd=\mathbb{R}_{\geq 0} v$ is given by 
    \[f_{\fd} = 
    \dfrac{\left(1+t^{2L_1^{v}}z^{-2v}\right)\left(1+t^{2L_2^{v}}z^{-2v}\right)}{\left(1-t^{C^{v}}z^{-2v}\right)^2}  \in \kk[NE(Y)] [\![ z^{-v} ]\!]\,.\] 
\end{lemma}
\begin{proof}
Since $SL(2,\mathbb{Z})$ acts transitively on primitive elements of $\mathbb{Z}^2$, $PSL(2,\mathbb{Z})$ acts transitively on primitive elements of $B(\mathbb{Z})$. Thus, there
exists $M \in PSL(2,\mathbb{Z})$ such that $M(v_1)=v$. Then, the result follows from Lemmas \ref{Lemma 4.4}, \ref{Lemma 4.5}, and the calculation of $f_{\rho_1}$
in Lemma \ref{Lemma 4.1}.
\end{proof}

\section{The intrinsic mirror algebra}
\label{Section 5}

In this section, we first define the intrinsic
mirror algebra from the canonical scattering diagram in \S\ref{sec_def_intrinsic}.
Then, in \S\ref{sec_calculation_intrinsic},
we prove Theorem \ref{thm_main_1}, stated as Theorem \ref{thm_main_intro_1} in the introduction, by explicitly determining the intrinsic mirror algebra of an orbifold cubic surface with three nodes. 

\subsection{Definition of the intrinsic mirror algebra}
\label{sec_def_intrinsic}

We first review the definition of broken lines in a scattering diagram.

\begin{defn}
	Let $\mathfrak{D}$ be a scattering diagram on the tropical cubic surface $(B, \Sigma)$. A \emph{broken line} $\gamma$ for $\mathfrak{D}$ with charge $p \in B(\Bbb{Z})\setminus \{0\}$ and endpoint $Q \in B\setminus \{0\}$ is a proper continuous piecewise integral affine map 
	$\gamma : (-\infty, 0 ] \to B$, together with real numbers $t_0=-\infty<t_1<\dots<t_n=0$, and monomials $m_1, \dots, m_n$, satisfying the following properties:
	\begin{itemize}
        \item[(i)] $\gamma(0) = Q \in B$.
		\item[(ii)] For every $1\leq i\leq n$, $\gamma|_{(t_{i-1},t_i]}$ is affine linear and $\gamma((t_{i-1}, t_i])$ is contained in a two-dimensional cone $\sigma_{j,j+1}$ of $\Sigma$. Moreover, 
    $m_i = c_i z^{-p_i}$, with $c_i \in \kk[NE(Y)]$ and $p_i = av_j+bv_{j+1} \neq 0$ for $a,b \in \mathbb{Z}$ such that $\gamma'(t)=p_i$ for any $t\in (t_{i-1}, t_i)$.
		\item[(iii)] $m_1=z^{-p}$.
		\item[(iv)] For every $1 \leq i < n$, either $\gamma(t_i)$ belongs to the support of a ray $(\fd, f_{\fd})$ or $\gamma(t_i) \in \rho_j$ for some $j$. If $\gamma(t_i)$ belongs to the support of a ray $(\fd, f_{\fd})$, then $\gamma$ passes from one side of $\fd$ to the other. Moreover, writing $\fd=\mathbb{R}_{\geq 0}v$ with primitive $v \in B(\mathbb{Z})$, 
        $m_{i+1}$ is a monomial in the power series expansion of 
        \[ m_i f_\fd^{|\det( v, p_i)|} \,.\]
        
        If $\gamma(t_i) \in \rho_j$ for some $j$, then $\gamma$ passes from one side of $\rho_j$ to the other. Moreover,
        $m_{i+1}$ is a monomial in the power series expansion of 
        \[ m_i t^{D_j} f_{\rho_j}^{|\det( v_j, p_i)|} \,.\]
	\end{itemize}
    We refer to $m_n =c_n z^{-p_n}$ as the \emph{final monomial} of $\gamma$, and denote it as $c(\gamma)z^{s(\gamma)}$, that is, $c(\gamma)=c_n$ and $s(\gamma)=-p_n$.
 \end{defn}

\begin{defn}
	Let $\mathfrak{D}$ be a scattering diagram on the tropical cubic surface 
    $(B, \Sigma)$. Let $p_1,p_2 \in B(\Bbb{Z})\setminus \{0\}$, $p \in B(\Bbb{Z})$ and $Q \in B \setminus \{0\}$. Define the \emph{structure constants}
	\[C^{\fD, p}_{p_1,p_2}(Q)  = \sum_{(\gamma_1 , \gamma_2)} c(\gamma_{1}) c(\gamma_{2}) \in \kk[\operatorname{NE}(Y)]\] where the sum is over pairs of broken lines $(\gamma_1, \gamma_2)$ with charges $p_1$, $p_2$, common endpoint $Q$, and final monomials $c(\gamma_{1})z^{s(\gamma_{1})}$ and $c(\gamma_{2})z^{s(\gamma_{2})}$ such that $s(\gamma_1) + s(\gamma_2) = p$.
	We extend this definition to all $p_1,p_2 \in B(\Bbb{Z})$ by $C^{\fD, p}_{0, p_2}(Q) = \delta_{p_2, p}$ and $C^{\mathfrak{D}, p}_{p_1, 0}(Q) = \delta_{p_1, p}$.
\end{defn}

\begin{definition}
	A scattering diagram $\mathfrak{D}$ is \emph{consistent} 
    if the following hold:
	\begin{itemize}
		\item[(i)] For every $p_1, p_2, p \in B(\Bbb{Z})$, the structure constant $C^{\mathfrak{D}, p}_{p_1,p_2}(Q)$ does not depend on $Q$, and is therefore denoted simply as $C^{\mathfrak{D}, p}_{p_1,p_2}$ from now on.
		\item[(ii)] The product on the $\kk[\operatorname{NE}(Y)]$-module 
		\[\mathcal{A}_{\mathfrak{D}} = \bigoplus_{p \in B(\Bbb{Z})}\kk[\operatorname{NE}(Y)] \vartheta_{p}\]
		defined by
		\[\vartheta_{p_1}\vartheta_{p_2} = \sum_{p \in B(\Bbb{Z})} C^{\mathfrak{D}, p}_{p_1,p_2} \vartheta_{p}\]
		is associative. 
	\end{itemize}
    When this is the case, we refer to $\mathcal{A}_{\mathfrak{D}}$ as the \emph{algebra} of $\fD$.
\end{definition}

\begin{thm} \label{thm_consistent}
Let $(Y,D)$ be an orbifold cubic surface with three nodes. Then, 
the associated canonical scattering diagram $\mathfrak{D}_{
\text{can}}$ is consistent. 
\end{thm}

\begin{proof}
There are two possible proofs. The first is to generalize the proof of \cite{GHK1}, adapting the arguments involving log Gromov--Witten invariants to the orbifold setting by using \cite[\S 5.5]{GPS}. Alternatively, one can use Remark \ref{remark_D11}, which identifies the quantum version of the canonical scattering diagram with the quantum scattering diagram $\fD_{1,1}$ of \cite{bousseau_skein}; the latter is consistent by \cite[Theorem 3.17]{bousseau_skein}, and so its classical limit is consistent too.
\end{proof}

\begin{definition}
    Let $(Y,D)$ be an orbifold cubic surface with three nodes. 
    The \emph{intrinsic mirror algebra} $R_{(Y,D)}$ is the algebra of the associated
    canonical scattering diagram $\mathfrak{D}_{\text{can}}$.  
\end{definition}

\subsection{Calculation of the intrinsic mirror algebra}
\label{sec_calculation_intrinsic}

In this section, we calculate the intrinsic mirror algebra $R_{(Y,D)}$ of an orbifold cubic surface with three nodes.

Let $(B, \Sigma)$ be the tropical cubic surface as in \S\ref{Section 2}.
Following \cite{bousseau_skein, GHKScubic}, we consider the unique continuous piecewise linear $F: B \rightarrow \Bbb{R}$ such that $F(v_i)=1$ for all $1 \leq i \leq 3$. According to  \cite[Proposition 2.7]{bousseau_skein}, for every broken line $\gamma$, the function $t \mapsto dF(\gamma'(t))$ is decreasing. Moreover, we have $F(r) \leq F(p_1)+F(p_2)$ if  $\vartheta_r$ appears in the product $\vartheta_{p_1}\vartheta_{p_2}$. 
These results are used in the proofs of the following lemmas to bound broken lines, as in \cite[\S 3]{GHKScubic}.

\begin{lemma}
    For every $\{i,j,k\}= \{1,2,3\}$, we have
	\[\vartheta^2_{v_i} = \vartheta_{2v_i} + 2 t^{D_j+D_k} \,.\]
    \label{Lemma 5.9}
\end{lemma}
\begin{proof}
We begin by establishing the result for $\vartheta_{v_1}^2$. 
The general result then follows by $\mathbb{Z}/3\mathbb{Z}$-symmetry.
Suppose that $\vartheta_r$ appears in the product $\vartheta_{v_1}^2$. Then
\[
F(r)\leq F(v_1)+F(v_1)=2.
\]
Moreover, since $F(s)\geq 0$ for every $s\in B(\mathbb{Z})$, it follows that
\[
0\leq F(r)\leq 2.
\]
We consider the possible values of $F(r)$ separately.

\begin{enumerate}
    \item The case $F(r)=2$. In this case, neither of the two broken lines $\gamma_1$ and $\gamma_2$, both carrying charge $v_1$, can undergo any bending. The only contribution arising in this way is $\vartheta_{2v_1}$.
        \begin{figure}[h!]
    \centering
    \includegraphics[width=5cm]{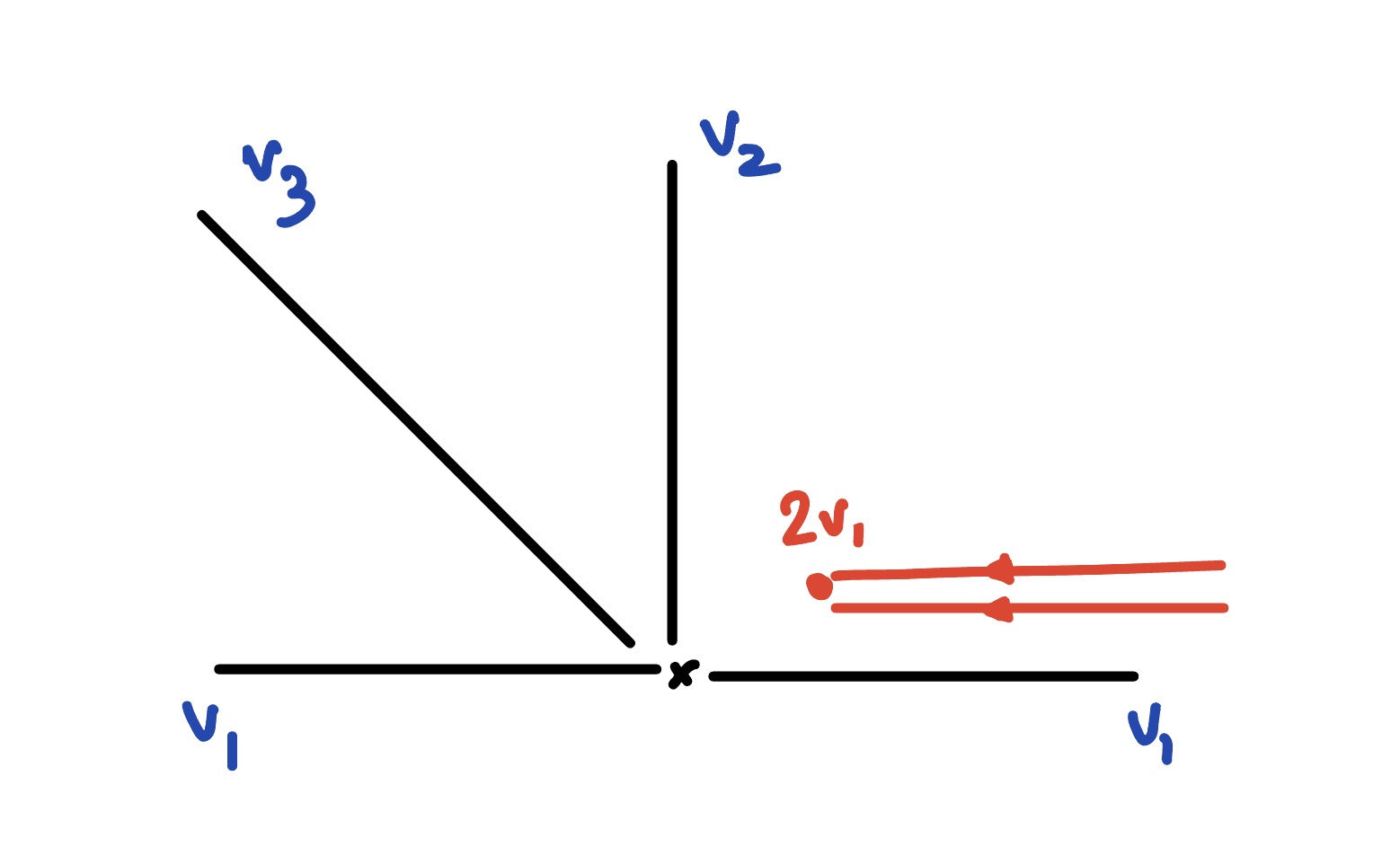}
    \caption{Contribution of $\vartheta_{2v_1}$ to $\vartheta_{v_1}^2$}
    \label{Figure 2}
    \end{figure}
        \item $F(r) = 1$:  This is the case in which no bending can occur, although one of the broken lines may cross a ray $\mathbb{R}_{\geq 0}v_i$ for some $i$.
Moreover, the only elements $s\in B(\mathbb{Z})$ satisfying $F(s)=1$ are $v_i$, for $i=1,2,3$. As shown in \cite{GHKScubic}, this case does not give rise to any contribution from broken lines.
        \item $F(r) = 0$: 
        The only element $s\in B(\mathbb{Z})$ for which $F(s)=0$ is the origin, $s=0$. In this case, no bending can occur. There are two possible configurations, each contributing $t^{D_2+D_3}\vartheta_0$. Hence, their total contribution is
$2t^{D_2+D_3}\vartheta_0=2t^{D_2+D_3}$. 
    
        \begin{figure}[h!]
    \centering
    \includegraphics[width=5cm]{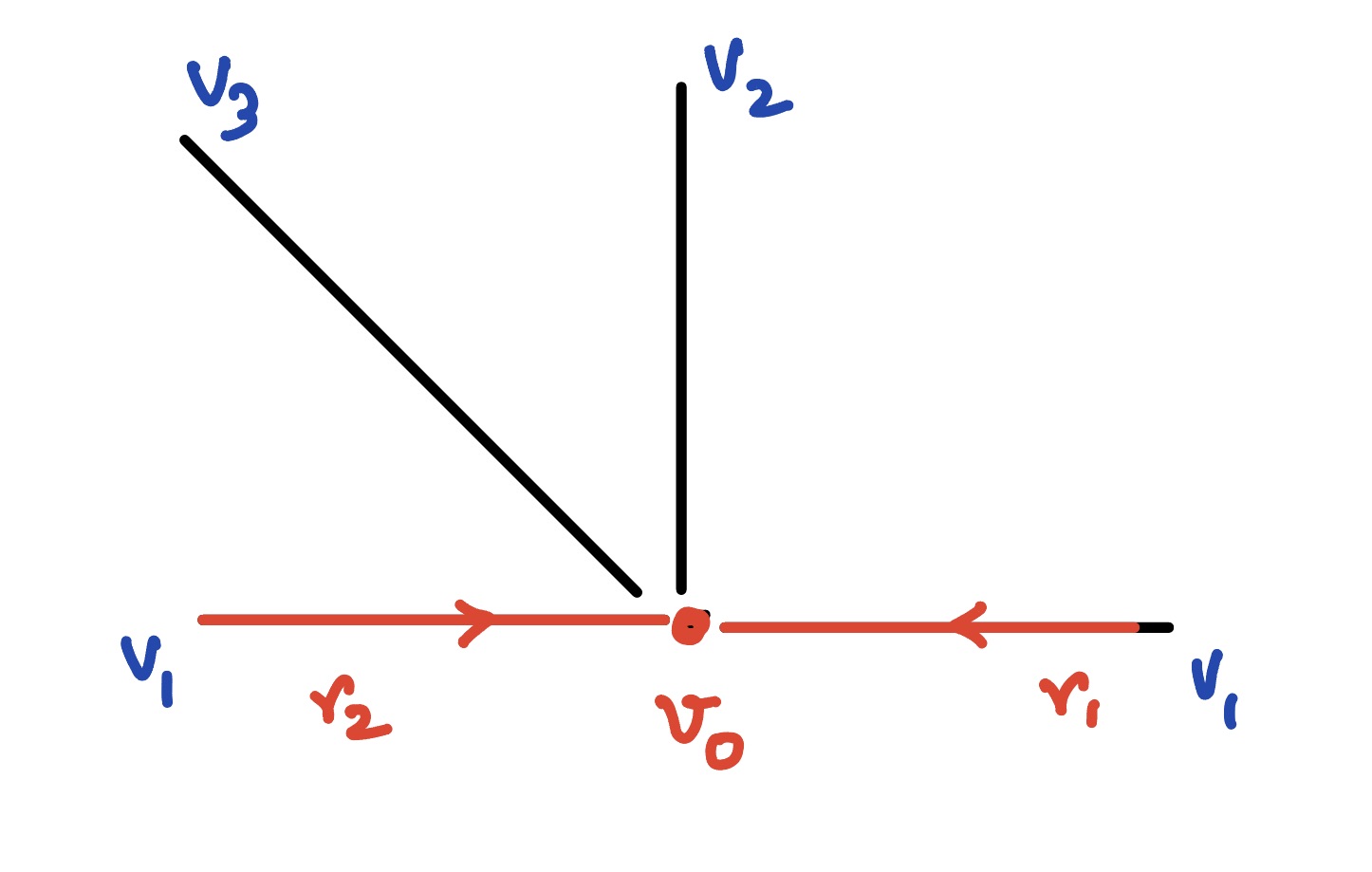}
     \caption{Contribution of $1=\vartheta_{0}$ to $\vartheta_{v_1}^2$}
     \label{Figure 3}
    \includegraphics[width=5cm]{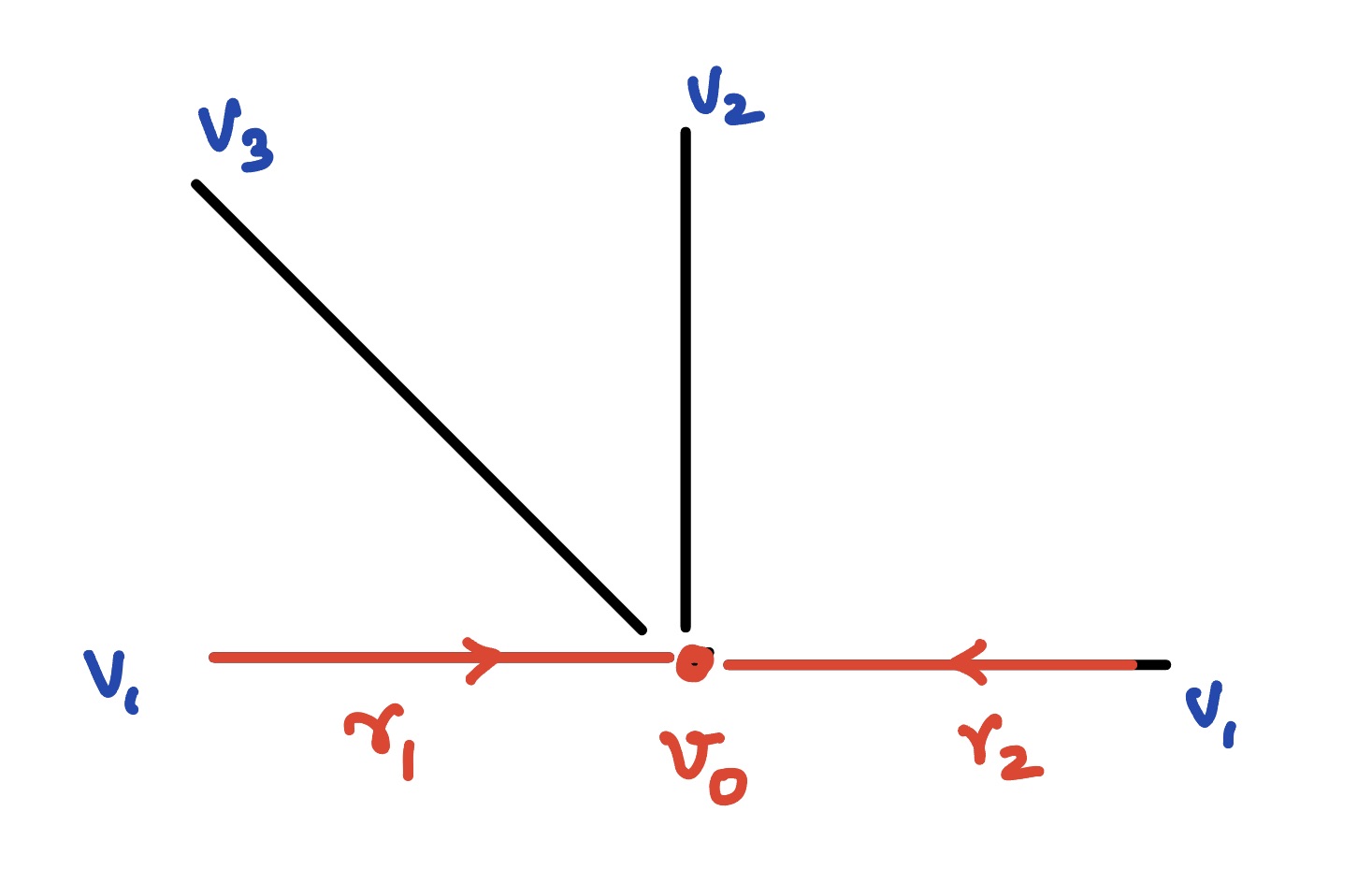}
     \caption{Contribution of $\vartheta_{0}$ to $\vartheta_{v_1}^2$}
     \label{Figure 4}
    \end{figure}
    \end{enumerate}
\end{proof}
\begin{lemma}
                   \[ \vartheta_{v_{1}} \vartheta_{v_{2}} = \vartheta_{v_1 + v_2} + t^{D_3} \vartheta_{v_3}\,. \]
                   \label{Lemma 5.10}
\end{lemma}
\begin{proof}
We necessarily have
$0\leq F(r)\leq 2$. 
As before, we distinguish three cases according to the value of $F(r)$:
\begin{enumerate}
    \item $F(r)=2$: In this case, no bending can occur. Consequently, the only possible contribution is $\vartheta_{v_1+v_2}$.
                \begin{figure}[h!]
    \centering
    \includegraphics[width=5cm]{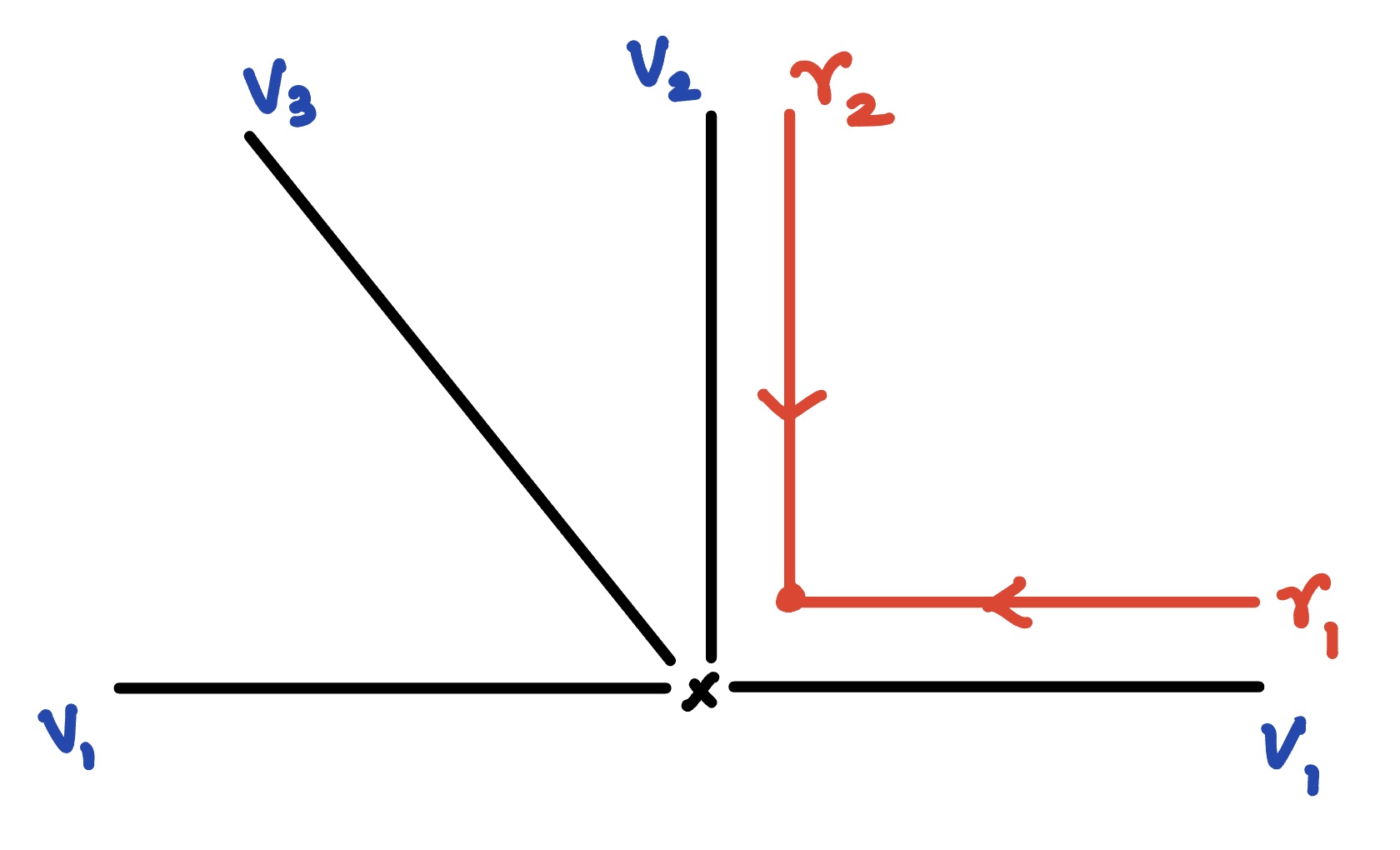}
     \caption{Contribution of $\vartheta_{v_1+v_2}$ to $\vartheta_{v_1}\vartheta_{v_2}$}
     \label{Figure 5}
    \end{figure}
        \item $F(r) = 1$: The only elements of $B(\Bbb{Z})$ satisfying this condition are $r = v_{1},v_{2}$ and $v_{3}$. Moreover, there cannot be any bending but broken lines can cross rays of $B$. This implies that the only contribution is from $\vartheta_{v_3}$, with coefficient $t^{D_3}$.
                \begin{figure}[h!]
    \centering
    \includegraphics[width=5cm]{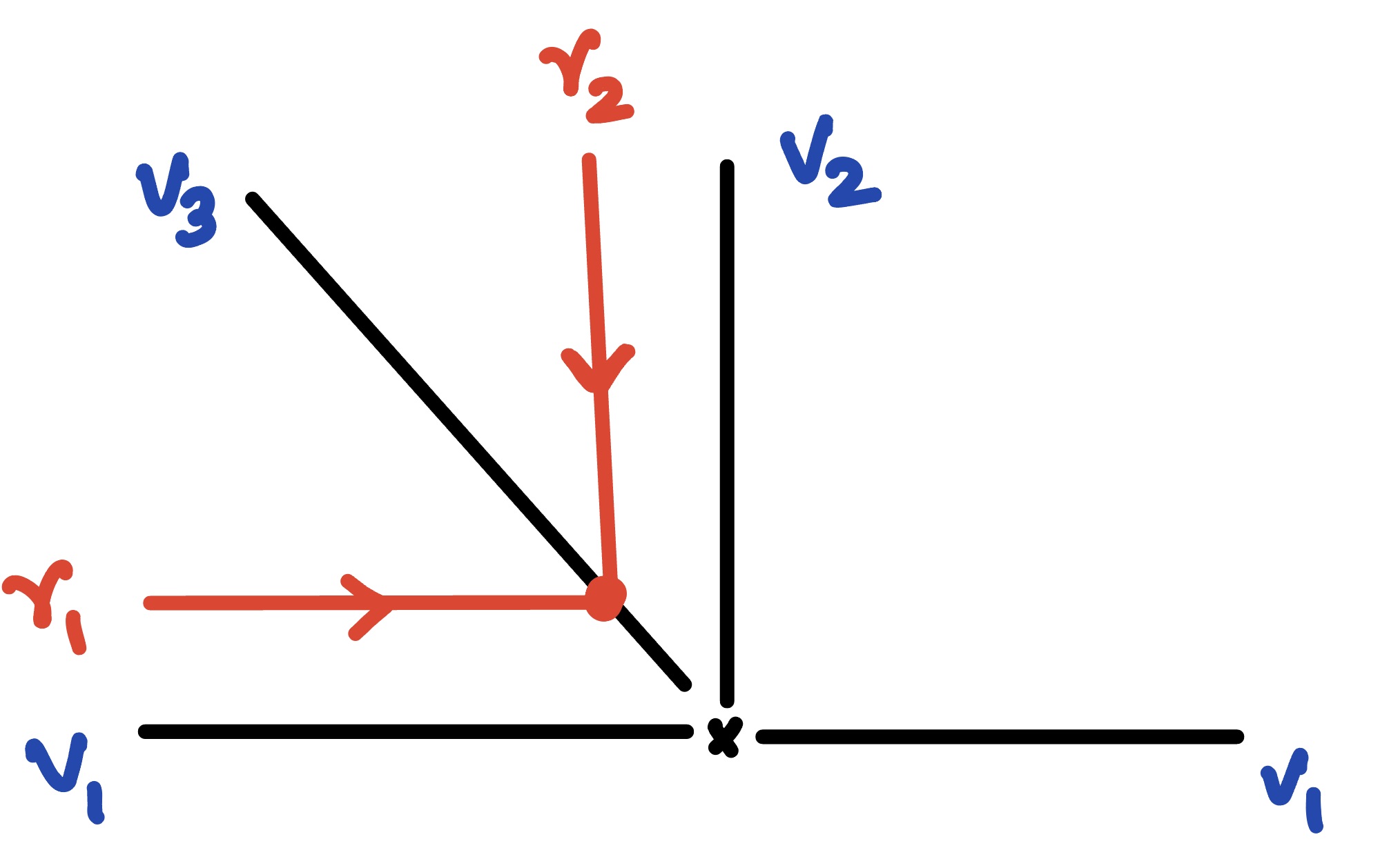}
      \caption{Contribution of $\vartheta_{v_3}$ to $\vartheta_{v_1}\vartheta_{v_2}$}
      \label{Figure 6}
       \end{figure}
        \item $F(r) = 0$: The only contribution can be from $\gamma_{1}$ along $\Bbb{R}_{\geq 0}v_2$ with charge $z^{(0,1)}$, but this implies a bending in the direction $(-2,1)$, which cannot happen. Hence there is no contribution from the case $F(r) = 0$. 
    \end{enumerate}
\end{proof}

\begin{lemma}
   $$\vartheta_{v_1 + v_2} \vartheta_{v_3}  = 
     t^{D_1} \vartheta_{2v_1} +  
    t^{D_2} \vartheta_{2v_2} + t^{H} + t^{2D-H} + 2 t^D \,.$$
   \label{Lemma 5.11}
\end{lemma}

\begin{proof}
As before, we have
$0\leq F(r)\leq 3$.
We consider the possible values of $F(r)$ separately:
\begin{enumerate}
    \item $F(r)=3$: This case gives no contribution, by the same argument as in \cite[Lemma 3.6]{GHKScubic}.

    \item $F(r)=2$: The only elements $r\in B(\mathbb{Z})$ satisfying this condition are
$r=2v_1, r=2v_2,r=2v_3$. 
No bending can occur, and so there are only two possible configurations, contributing $t^{D_1}\vartheta_{2v_1}$ and $t^{D_2}\vartheta_{2v_2}$, respectively.
        \begin{figure}[h!]
    \centering
    \includegraphics[width=9cm]{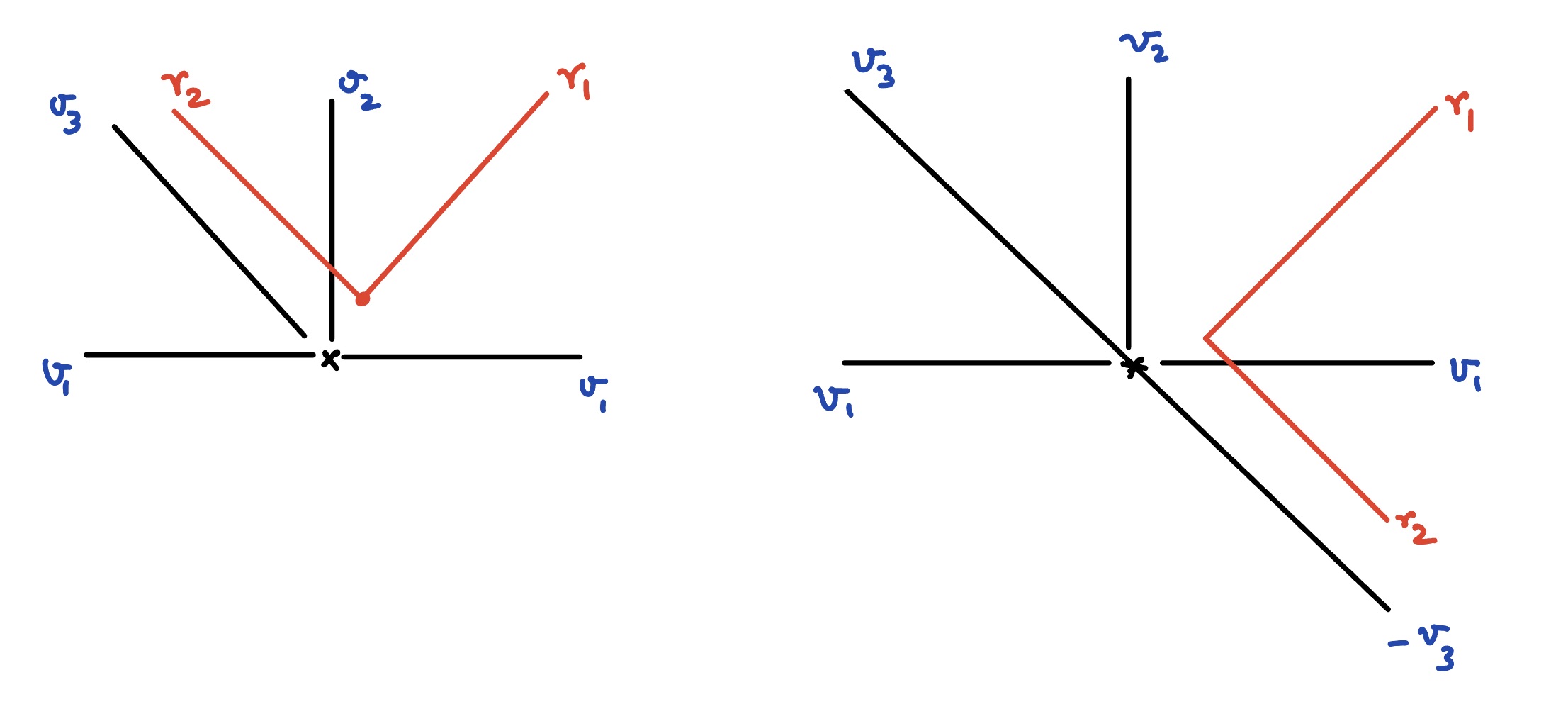}
    \caption{Contribution of $\vartheta_{2v_2}$ and $\vartheta_{2v_1}$ to $\vartheta_{v_1+v_2}\vartheta_{v_3}$}
    \label{Figure 7}
    \end{figure}
        \item $F(r) = 1$: There are no contributions in this case. 
        \item $F(r) = 0$: The only contribution is illustrated in the figure below. Using Lemma \ref{Lemma 4.1} for the function attached to $\rho_1$,  we obtain 
\[ t^{D_1} (t^{2L_1^{v_1}} + t^{2L_2^{v_1}} + 2 t^{C^{v_1}}) = t^{D_1+2E_1} +t^{D_1+4H-4E_2-4E_3-2E_1}+2t^{D_1+2H-2E_2-2E_3} \,,\]
        which, using $D_1=H-2E_1$, $D_2=H-2E_2$, $D_3=H-2E_3$, and $D=D_1+D_2+D_3=3H-2E_1-2E_2-2E_3$, equals
        \[t^{H} + t^{2D-H} + 2 t^D\,.\]
        \begin{figure}[h!]
    \centering
    \includegraphics[width=7cm]{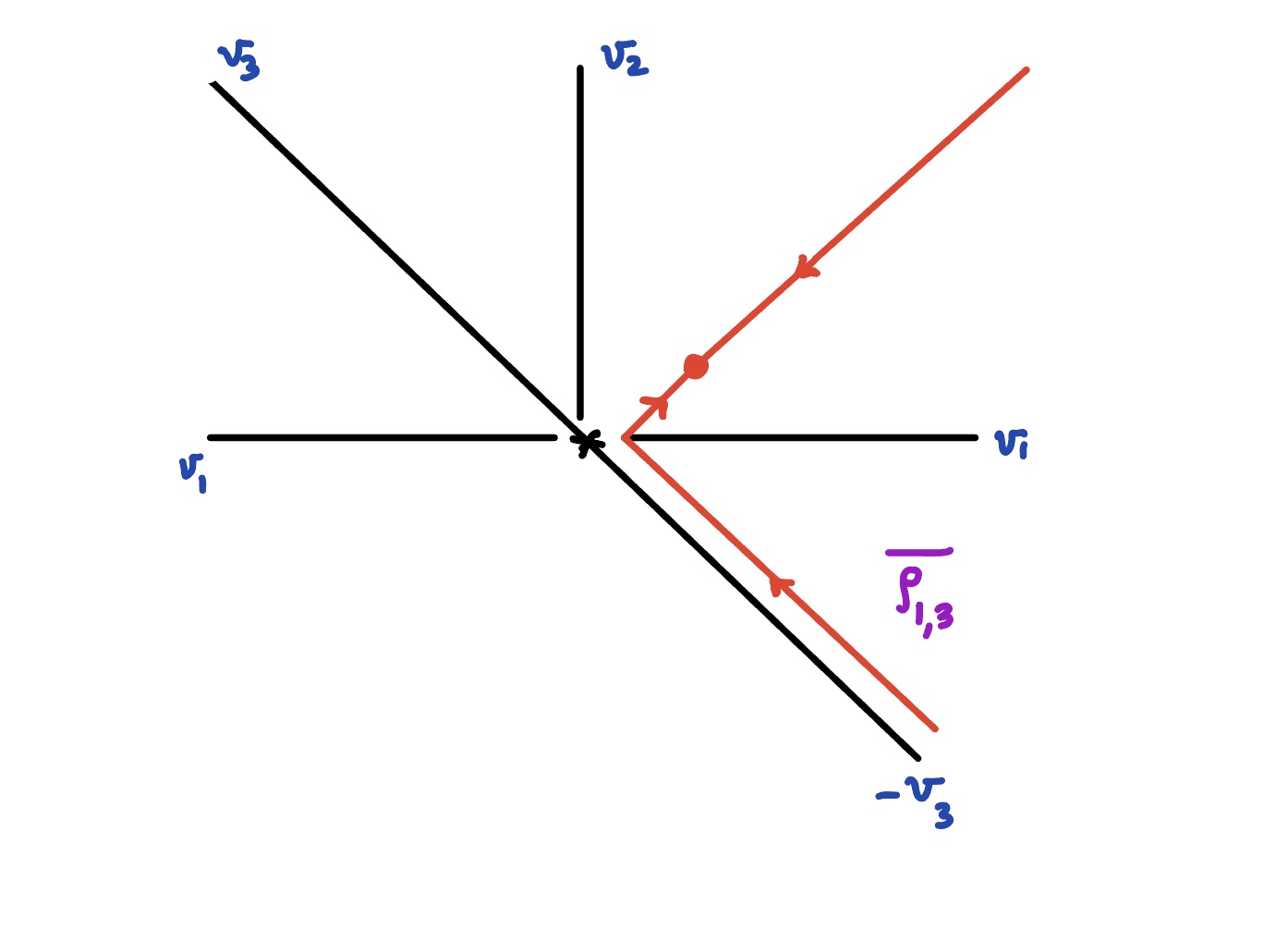}
    \caption{Contribution of $1=\var_0$ to $\vartheta_{v_1+v_2}\vartheta_{v_3}$}
    \label{Figure 8}
    \end{figure}
    \end{enumerate}
\end{proof}

Finally, the following Theorem \ref{thm_main_1}, stated as Theorem \ref{thm_main_intro_1} in the introduction, determines 
the intrinsic mirror algebra of an orbifold cubic surface with three nodes.

\begin{thm}
\label{thm_main_1}
Let $(Y,D)$ be an orbifold cubic surface over $\kk$, where $Y$ is a projective cubic surface with three nodes and $D \subset Y$ is a triangle of lines avoiding the nodes. 
The intrinsic mirror algebra
$R_{(Y,D)}$
of $(Y,D)$ 
is the commutative $\kk[\operatorname{NE}(Y)]$-algebra generated by three elements 
$\vartheta_{v_1}, \vartheta_{v_2}, \vartheta_{v_3}$ that satisfy the relation
    \[\var_{v_1} \var_{v_2} \var_{v_3} =t^{D_1} \var_{v_1}^2 +t^{D_2} \var_{v_2}^2 +t^{D_3} \var_{v_3}^2 + t^H + t^{2D-H} - 2t^D \,.\]
\end{thm}

\begin{proof}
Using Lemmas \ref{Lemma 5.9}, \ref{Lemma 5.10}, and \ref{Lemma 5.11}, we obtain
\begin{align*}
\var_{v_1}\var_{v_2}\var_{v_3} & =  \left(\vartheta_{v_1+v_2} + t^{D_3} \vartheta_{v_3}\right)  \var_{v_3} \\
& = \var_{v_1+v_2} \var_{v_3} +t^{D_3} \var_{v_3}^2 \\
& = t^{D_1}\vartheta_{2v_1} +  t^{D_2}\vartheta_{2v_2} + t^H + t^{2D-H} + 2t^D + t^{D_3}\var_{v_3}^2 \\
& = t^{D_1}(\vartheta_{v_1}^2 - 2 t^{D_2+D_3}) +  t^{D_2}(\vartheta_{v_2}^2-2t^{D_1+D_3}) + t^H + t^{2D-H} + 2t^D + t^{D_3}\var_{v_3}^2\\
& = t^{D_1}\var_{v_1}^2 + t^{D_2}\var_{v_2}^2 + t^{D_3}\var_{v_3}^2 + t^H + t^{2D-H} -2 t^D \,.
 \end{align*}
\end{proof}

\section{The quantum canonical scattering diagram}
\label{Section 6}

In \S \ref{sec_def_canonical_quantum}, we define the quantum canonical scattering diagram for orbifold cubic surfaces. The main result of this section is the explicit calculation of this canonical scattering diagram in \S\ref{sec_calculation_canonical_ray}-\ref{sec_calculation_canonical_general_ray} 
 by direct evaluation of higher genus orbifold log Gromov--Witten invariants.

\subsection{Definition of the quantum canonical scattering diagram} \label{sec_def_canonical_quantum}

Let $(B,\Sigma)$ be the tropical cubic surface defined in \S\ref{Section 2}. 
Denote $\kk_\hbar:= \kk[\![\hbar]\!]$ and by
$\kk_\hbar[NE(Y)]$
the monoid algebra of $NE(Y)$ with coefficients in $\kk_\hbar$. As in \S\ref{sec_def_canonical},
an ideal $I \subset \kk_\hbar[NE(Y)]$ is called co-Artinian if the quotient $\kk_\hbar[NE(Y)]/I$ is an Artinian $\kk_\hbar$-algebra.

\begin{defn} A \emph{quantum ray} in $(B, \Sigma)$ is a pair $(\mathfrak{d}, \widehat{f}_{\mathfrak{d}})$, where \begin{itemize}
    \item[(i)] $\mathfrak{d} \subset B$ is a ray $\mathbb{R}_{\geq 0}v_\fd$ generated by some primitive $v_{\fd} \in B(\Bbb{Z}) \setminus \{0\}$.
    \item[(ii)] $\widehat{f}_{\mathfrak{d}} \in \kk_\hbar[NE(Y)][\![z^{-v_\fd}]\!]$ such that $\widehat{f}_{\mathfrak{d}} =  1 \operatorname{mod}z^{-v_\fd}$, and for every co-Artinian ideal $I \subset \kk_\hbar[NE(Y)]$, $\widehat{f}_{\mathfrak{d}}\, \mathrm{mod}\, I$ is a finite sum.
\end{itemize}    
\end{defn}
\begin{defn}
    A \emph{quantum scattering diagram} $\widehat{\fD}$ is a collection of rays $(\mathfrak{d}, \widehat{f}_{\mathfrak{d}})$ such that:
   \begin{itemize} 
\item[(i)]    $\Bbb{R}_{\geq 0}v_{\mathfrak{d}} = \Bbb{R}_{\geq 0}v_{\mathfrak{d}'}$ implies $\mathfrak{d}=\mathfrak{d}'$,
\item[(ii)] for every co-Artinian ideal $I \subset \kk_\hbar[NE(Y)]$, there are only finitely many rays $(\mathfrak{d}, \widehat{f}_{\mathfrak{d}})$ such that $\widehat{f}_{\mathfrak{d}} \neq 1\, \mathrm{mod}\, I$.
    \end{itemize}
\end{defn}

Let $(Y,D)$ be an orbifold cubic surface with three nodes. Fix a curve class $\beta \in \operatorname{NE}(Y)$ and an integral point $v \in B(\Bbb{Z})$. 
As in \S\ref{sec_def_canonical}, we consider the moduli space 
$\overline{\mathcal{M}}^{\beta}_{g,v}(Y,D)$
of genus $g$ orbifold basic stable log maps to $(Y,D)$ with one marked point of contact order $v$ along $D$.
It is a proper Deligne--Mumford stack and carries a $g$-dimensional virtual class $[\overline{\mathcal{M}}^{\beta}_{g,v}(Y, D)]^{\mathrm{vir}}$.
As in \cite{bousseau_mirror,  bousseau_vertex,  bousseau_skein}, we define genus $g$ orbifold log Gromov--Witten invariants 
\[ N_{g,v}^\beta := \int_{[\overline{\mathcal{M}}^{\beta}_{g,v}(Y, D)]^{\mathrm{vir}}} (-1)^g \lambda_g \in \mathbb{Q} \,,\]
where $\lambda_g$ is the top Chern class of the Hodge bundle $\pi_\star \omega_\pi$, where $\pi: \mathcal{C} \rightarrow \overline{\mathcal{M}}^{\beta}_{g,v}(Y, D)$
is the universal domain curve, and $\omega_\pi$ is the relative dualizing line bundle.

We can now define the quantum canonical scattering diagram associated with $(Y,D)$ as in \cite{bousseau_mirror}, with the only modification arising from the orbifold setting.

\begin{defn}
Let $(Y,D)$ be an orbifold cubic surface with three nodes. 
    The \emph{quantum canonical scattering diagram} $\widehat{\mathfrak{D}}_{\text{can}}$ of $(Y,D)$ is the collection of rays $(\mathfrak{d}, \widehat{f}_{\mathfrak{d}})$, where $\mathfrak{d} = \Bbb{R}_{\geq 0}v$ with primitive $v \in B(\mathbb{Z})$ and 
    \[\widehat{f}_{\mathfrak{d}} 
    = \operatorname{exp}\left(\sum_{k\geq 1}\sum_{\beta \in NE(Y)} \left( 2 \sin \left( \frac{k \hbar}{2} \right) \right) \sum_{g \geq 0} N_{g, kv}^\beta \hbar^{2g-1} t^{\beta}z^{-kv}\right) \in \kk_\hbar[NE(Y)][\![z^{-v}]\!]\,.\]
\end{defn}

\subsection{The ray $(\rho_1, \widehat{f}_{\rho_1})$ in the quantum canonical scattering diagram}
\label{sec_calculation_canonical_ray}
In this section, we calculate the function $\widehat{f}_{\rho_1}$ associated to the ray $\rho_1=\mathbb{R}_{\geq 0}v_1$ of the quantum canonical 
scattering diagram.

\begin{lemma} \label{lem_line_g}
We have 
\[ \exp \left( \sum_{k \geq 1} \left(
2 \sin \left( \frac{k \hbar}{2} \right)
\right) \sum_{g \geq 0} N_{g, kv_1}^{k L_1^{v_1}} \hbar^{2g-1} t^{k L_1^{v_1}} z^{-kv_1}\right) = 1 + t^{2 L_1^{v_1}} z^{-2v_1} \]
and 
\[ \exp \left( \sum_{k \geq 1} 
\left(
2 \sin \left( \frac{k \hbar}{2} \right)
\right)  \sum_{g \geq 0}N_{g, kv_1}^{k L_2^{v_1}} \hbar^{2g-1}t^{k L_2^{v_1}} z^{-kv_1}\right) = 1 + t^{2 L_2^{v_1}} z^{-2v_1}\,. \]
\end{lemma}

\begin{proof}
By Lemmas \ref{lem_classification} and \ref{lem_lines}, every orbifold stable log map contributing to $N_{g, kv_1}^{k L_1^{v_1}}$ factors through the unique line of class $L_1^{v_1}$. This line contains exactly one $\mathbb{Z}/2\mathbb{Z}$ orbifold point and intersects $D$ transversely at a single point. Therefore, the calculation of $N_{g, kv_1}^{k L_1^{v_1}}$ reduces to a calculation in relative orbifold Gromov-Witten theory of the stacky $\mathbb{P}^1$ with a $\mathbb{Z}/2\mathbb{Z}$ orbifold point, and so \cite[Lemma 7.2]{bousseau_vertex} implies that $N_{g, kv_1}^{k L_1^{v_1}}=0$ if $k$ is odd, and
\[ \sum_{g \geq 0 }N_{g, k v_1}^{k L_1^{v_1}} \hbar^{2g-1} = \frac{(-1)^{l-1}}{l} \frac{1}{2\sin (l \hbar)}\] if $k=2l$ is even. 
Therefore, we obtain
\[ \exp \left( \sum_{k \geq 1} \left(
2 \sin \left( \frac{k \hbar}{2} \right)
\right) \sum_{g \geq 0} N_{g, kv_1}^{k L_1^{v_1}} \hbar^{2g-1} t^{k L_1^{v_1}} z^{-kv_1}\right) \] 
\[=\exp \left(  \sum_{l \geq 1}  \frac{(-1)^{l-1}}{l} t^{2l L_1^{v_1}} z^{-2l v_1}\right)
=
1 + t^{2 L_1^{v_1}} z^{-2v_1} \,.\]
The same argument applies for $L_1^{v_1}$ replaced by $L_2^{v_1}$.
\end{proof}

\begin{lemma} \label{lem_conic_g}
We have \begin{equation} \label{eq_proof}
\exp \left( \sum_{l \geq 1} (2\sin(l \hbar)) \sum_{g \geq 0}N_{g, 2lv_1}^{l C^{v_1}} \hbar^{2g-1} t^{l C^{v_1}} z^{-2lv_1}\right) =
\frac{1}{(1 - qt^{C^{v_1}} z^{-2v_1}) (1 - q^{-1}t^{C^{v_1}} z^{-2v_1})}\,, \end{equation}
where $q=e^{i\hbar}$.
\end{lemma}

\begin{proof}
It follows from Lemma \ref{lem_conics} that the orbifold stable log maps contributing to $N_{g, 2lv_1}^{l C^{v_1}}$ do not contain the node $x_1$. Thus, these orbifold log Gromov--Witten invariants can be calculated using the log Calabi--Yau surface with two nodes $x_2$ and $x_3$, obtained from $\mathbb{P}^2$ by blowing-up successively two points instead of three. By \cite[Theorem 3.2]{bousseau_vertex}, these orbifold log Gromov--Witten invariants
are calculated by a quantum scattering diagram in $\mathbb{R}^2$. Specifically, the left-hand side of \eqref{eq_proof} is the function attached to the ray of direction $\mathbb{R}_{\geq 0}(e_1+e_2)$ in the consistent quantum scattering diagram in $\mathbb{R}^2$ with incoming rays 
\[
\left(
\mathbb R_{\leq 0}e_i,\,
1+z^{2e_i}
\right),
\qquad i=1,2\,,
\]
where $e_1, e_2$ form a basis of $\mathbb{Z}^2 \subset \mathbb{R}^2$, with $\det(e_1, e_2)=1$.
On the other hand, the quantum stability scattering diagram of the Kronecker quiver $Q_2$, with two vertices connected by two arrows, has initial rays  
\[
\left(
\mathbb R_{\leq 0}e_i',\,
1+z^{e'_i}
\right),
\qquad i=1,2\,,
\]
with $\det(e_1', e_2')=2$, and the ray of direction $\mathbb{R}_{\geq 0}(e_1'+e_2')$
has attached function 
\[ \frac{1}{(1-(q')^{-\frac{1}{2}} z^{e_1'+e_2'})(1-(q')^{\frac{1}{2}} z^{e_1'+e_2'})}\]
since the Donaldson--Thomas invariant of the quiver with dimension vector $(1,1)$ is $\Omega = -(q')^{1/2}-(q')^{-1/2}$ -- see \cite[Example 6.10]{mandel_positivity}.
The two quantum scattering diagrams are related by the change of lattice and parameters
\[
e_i'\longmapsto 2e_i,
\qquad
\hbar'\longmapsto 2\hbar,
\qquad
q'\longmapsto q^2\,,
\]
and so the result follows.
\end{proof}

\begin{lemma}
	\label{Lemma_ray_quantum}
    The ray $(\rho_1, \widehat{f}_{\rho_1})$ in the quantum canonical scattering diagram $\widehat{\mathfrak{D}}_{\text{can}}$ is given by 
    \[\widehat{f}_{\rho_{1}} = 
    \dfrac{\left(1+t^{2L_1^{v_1}}z^{-2v_1}\right)\left(1+t^{2L_2^{v_1}}z^{-2v_1}\right)}{\left(1-qt^{C^{v_1}}z^{-2v_1}\right)\left(1-q^{-1}t^{C^{v_1}}z^{-2v_1}\right)} \,,\] 
    where $L_1^{v_1} = E_1, L_2^{v_1} =2H- 2E_2-2E_3- E_1$ and $C^{v_1} = 2H - 2E_2 - 2E_3$.
\end{lemma}

\begin{proof}
By deformation invariance of orbifold log Gromov--Witten invariants, one can assume that $Y$ is general enough.
By Lemmas \ref{lem_classification}, \ref{Lemma 4.3}, together with Lemmas \ref{lem_lines},\ref{lem_bad_line},\ref{lem_conics}, the only curve classes contributing to $\widehat{f}_{\rho_1}$ are multiples of $L_1^{v_1}$, $L_2^{v_1}$, and $H-E_2-E_3$. The contributions of these curve classes are determined by 
Lemmas \ref{lem_line_g} and \ref{lem_conic_g}.
\end{proof}

\subsection{General rays in the quantum canonical scattering diagram}
\label{sec_calculation_canonical_general_ray}

Recall from \S\ref{Subsection 4.1} that there is a natural action of $PSL(2, \mathbb{Z})$ on $B(\mathbb{Z})$ and $A_1(Y)$.

\begin{lemma}
    For every primitive $v \in B(\Bbb{Z})$, there exists $M \in PSL(2,\Bbb{Z})$, such that, denoting $L_1^{v}=M_\star(L_1^{v_1})$, 
    $L_2^{v}=M_\star(L_2^{v_1})$, $C^{v}=M_\star(C^{v_1})$, the ray $(\fd, \widehat{f}_\fd)$ of 
    the quantum canonical scattering diagram $\widehat{\fD}_{\text{can}}$ with $\fd=\mathbb{R}_{\geq 0} v$ is given by 
    \[\widehat{f}_{\fd} = 
    \dfrac{\left(1+t^{2L_1^{v}}z^{-2v}\right)\left(1+t^{2L_2^{v}}z^{-2v}\right)}{\left(1-qt^{C^{v}}z^{-2v}\right)\left(1-q^{-1}t^{C^{v}}z^{-2v}\right)}  \in \kk_\hbar[NE(Y)] [\![ z^{-v} ]\!]\,,\]
    where $q=e^{i\hbar}$.
\end{lemma}
\begin{proof}
This follows as in \S\ref{Subsection 4.1} from the transitive action of $PSL(2, \mathbb{Z})$ on the primitive elements of $B(\mathbb{Z})$, and from the calculation of $\widehat{f}_{\rho_1}$
in Lemma \ref{Lemma_ray_quantum}.
\end{proof}

\begin{rmk} \label{remark_D11}
After intersecting curve classes with $\sD = 2H-2E_1-2E_2-2E_3$, the function $\widehat{f}_\fd$ becomes
\[  \dfrac{\left(1+t^2 z^{-2v}\right)\left(1+t^{-2} z^{-2v}\right)}{\left(1-q z^{-2v}\right)\left(1-q^{-1} z^{-2v}\right)}  \,,\]
and so $\widehat{\fD}_{\text{can}}$ reduces to the quantum scattering diagram $\fD_{1,1}$ in \cite[\S 3.2]{bousseau_skein}
after the identification $\lambda = t^2$. In \cite{bousseau_skein}, $\fD_{1,1}$ is described as a base change of a scattering diagram
$\fD_{0,4}$ which is realized as the quantum canonical scattering diagram of a smooth cubic surface. The present paper provides a more direct geometric interpretation of 
$\fD_{1,1}$ as the quantum canonical scattering diagram of an orbifold cubic surface.
\end{rmk}

\section{The quantum intrinsic mirror algebra}
\label{Section 7}

In this section, we prove Theorem \ref{thm_main_2}, stated as Theorem \ref{thm_main_intro_2} in the introduction, by explicitly determining the quantum intrinsic mirror algebra of an orbifold cubic surface with three nodes. We begin in \S\ref{sec_quantum_intrinsic} by reviewing the definition of the quantum intrinsic mirror algebra, before carrying out its computation in \S\ref{sec_calculation_quantum}.

\subsection{The quantum intrinsic mirror algebra}
\label{sec_quantum_intrinsic}

We first review the definition of quantum broken lines and quantum theta functions given in \cite{bousseau_vertex, bousseau_skein}.
\begin{defn}
	Let $\widehat{\mathfrak{D}}$ be a quantum scattering diagram on the tropical cubic surface $(B, \Sigma)$. A \emph{quantum broken line} $\gamma$ for $\widehat{\mathfrak{D}}$ with charge $p \in B(\Bbb{Z})\setminus \{0\}$ and endpoint $Q \in B\setminus \{0\}$ is a proper continuous piecewise integral affine map 
	$\gamma : (-\infty, 0 ] \to B$, together with real numbers $t_0=-\infty<t_1<\dots<t_n=0$, and monomials $m_1, \dots, m_n$, satisfying the following properties:
	\begin{itemize}
        \item[(i)] $\gamma(0) = Q \in B$.
		\item[(ii)] For every $1\leq i\leq n$, $\gamma|_{(t_{i-1},t_i]}$ is affine linear and $\gamma((t_{i-1}, t_i])$ is contained in a two-dimensional cone $\sigma_{j,j+1}$ of $\Sigma$. Moreover, 
    $m_i = c_i z^{-p_i}$, with $c_i \in \kk_q[NE(Y)]$ and $p_i = av_j+bv_{j+1} \neq 0$ for $a,b \in \mathbb{Z}$ such that $\gamma'(t)=p_i$ for any $t\in (t_{i-1}, t_i)$.
		\item[(iii)] $m_1=z^{-p}$.
		\item[(iv)] For every $1 \leq i < n$, either $\gamma(t_i)$ belongs to the support of a quantum ray $(\fd, \widehat{f}_{\fd})$ or $\gamma(t_i) \in \rho_j$ for some $j$. 
        
        If $\gamma(t_i)$ belongs to the support of a quantum ray $(\fd, \widehat{f}_{\fd})$, then $\gamma$ passes from one side of $\fd$ to the other. Moreover, writing $\fd=\mathbb{R}_{\geq 0}v$ with primitive $v \in B(\mathbb{Z})$, and $\widehat{f}_{\fd} = \sum_{k \geq 0 }H_{k}z^{-kv} \in \kk_q[NE(Y)][\![z^{-v}]\!]$, there exists a monomial $d_l z^{-lv}$ in the series \[\sum_{l\geq 0} d_{l}z^{-lv} = \prod_{m=0}^{|\det( v, p_i)| -1}\sum_{k \geq 0}(q^{m - \frac{1}{2}(|\det(v, p_i)| -1)})^{k}H_{k}z^{-kv}\]
        such that $c_{i+1}= d_l c_i$ and $p_{i+1}=p_i + l v$.

        If $\gamma(t_i) \in \rho_j$ for some $j$, then $\gamma$ passes from one side of $\rho_j$ to the other. Moreover, 
        writing $\widehat{f}_{\fd} = \sum_{k \geq 0 }H_{k}z^{-kv_j} \in \kk_q[NE(Y)][\![z^{-v_j}]\!]$, there exists a monomial $d_l z^{-lv_j}$ in the series \[\sum_{l\geq 0} d_{l}z^{-lv_j} = \prod_{m=1}^{|\det( v_j, p_i)| }\sum_{k \geq 0} (q^{m - \frac{1}{2}(|\det(v_j, p_i) | -1)})^{k}H_{k}z^{-kv_j}\]
        such that $c_{i+1}= t^{D_j} d_l c_i$ and $p_{i+1}=p_i + l v$.
	\end{itemize}
    We refer to $m_n =c_n z^{-p_n}$ as the \emph{final monomial} of $\gamma$, and denote it as $c(\gamma)z^{s(\gamma)}$, that is, $c(\gamma)=c_n$ and $s(\gamma)=-p_n$.
 \end{defn}

\begin{defn}
	Let $\widehat{\mathfrak{D}}$ be a quantum scattering diagram on the tropical cubic surface 
    $(B, \Sigma)$. Let $p_1,p_2 \in B(\Bbb{Z})\setminus \{0\}$, $p \in B(\Bbb{Z})$ and $Q \in B \setminus \{0\}$. Define the \emph{structure constants}
	\[C^{\widehat{\fD}, p}_{p_1,p_2}(Q)  = \sum_{(\gamma_1 , \gamma_2)} c(\gamma_{1}) c(\gamma_{2})q^{\frac{1}{2}\det ( s(\gamma_{1}), s(\gamma_2))} \in \kk_q[\operatorname{NE}(Y)]\] where the sum is over pairs of quantum broken lines $(\gamma_1, \gamma_2)$ with charges $p_1$, $p_2$, common endpoint $Q$, and final monomials $c(\gamma_{1})z^{s(\gamma_{1})}$ and $c(\gamma_{2})z^{s(\gamma_{2})}$ such that $s(\gamma_1) + s(\gamma_2) = p$.
	We extend this definition to all $p_1,p_2 \in B(\Bbb{Z})$ by $C^{\widehat{\fD}, p}_{0, p_2}(Q) = \delta_{p_2, p}$ and $C^{\widehat{\fD}, p}_{p_1, 0}(Q) = \delta_{p_1, p}$.
\end{defn}

\begin{definition}
	A quantum scattering diagram $\widehat{\mathfrak{D}}$ is \emph{consistent} 
    if the following hold:
	\begin{itemize}
		\item[(i)] For every $p_1, p_2, p \in B(\Bbb{Z})$, the structure constant $C^{\widehat{\fD}, p}_{p_1,p_2}(Q)$ does not depend on $Q$, and is therefore denoted simply as $C^{\widehat{\fD}, p}_{p_1,p_2}$ from now on.
		\item[(ii)] The product on the $\kk_q[\operatorname{NE}(Y)]$-module 
		\[\mathcal{A}_{\widehat{\mathfrak{D}}} = \bigoplus_{p \in B(\Bbb{Z})}\kk_q[\operatorname{NE}(Y)] \widehat{\vartheta}_{p}\]
		defined by
		\[\widehat{\vartheta}_{p_1}\widehat{\vartheta}_{p_2} = \sum_{p \in B(\Bbb{Z})} C^{\widehat{\fD}, p}_{p_1,p_2} \widehat{\vartheta}_{p}\]
		is associative. 
	\end{itemize}
    When this is the case, we refer to $\mathcal{A}_{\widehat{\mathfrak{D}}}$ as the \emph{algebra} of $\widehat{\fD}$.
\end{definition}

\begin{thm} \label{thm_quantum_consistent}
Let $(Y,D)$ be an orbifold cubic surface with three nodes. Then, 
the associated quantum canonical scattering diagram $\widehat{\mathfrak{D}}_{
\text{can}}$ is consistent. 
\end{thm}

\begin{proof}
We give two ways to deduce consistency. The first is to generalize the proof of \cite[Theorem 4.4]{bousseau_mirror}, adapting the arguments involving log Gromov--Witten invariants to the orbifold setting by using \cite[\S 2.5]{bousseau_vertex}. Alternatively, one can use Remark \ref{remark_D11}, which identifies the quantum canonical scattering diagram with the quantum scattering diagram $\fD_{1,1}$ of \cite{bousseau_skein}; the latter is consistent by \cite[Theorem 3.17]{bousseau_skein}.
\end{proof}

\begin{definition}
    Let $(Y,D)$ be an orbifold cubic surface with three nodes. 
    The \emph{quantum intrinsic mirror algebra} $\widehat{R}_{(Y,D)}$ is the algebra of the associated
    quantum canonical scattering diagram $\widehat{\mathfrak{D}}_{\text{can}}$.  
\end{definition}

\subsection{Calculation of the quantum intrinsic mirror algebra}
\label{sec_calculation_quantum}

We now obtain the quantum analogues of Lemmas \ref{Lemma 5.9}, \ref{Lemma 5.10}, and \ref{Lemma 5.11}.

\begin{lemma}
	\label{Lemma 6.5}
	For every $\{i,j,k\}= \{1,2,3\}$, we have
	\[\widehat{\vartheta}^2_{v_i} = \widehat{\vartheta}_{2v_i} + 2 t^{D_j+D_k} \,.\]
\end{lemma}
\begin{proof}
By the cyclic $\Bbb{Z}/3\Bbb{Z}$ symmetry, it suffices to show the result for $i=1$. We follow the proof of Lemma \ref{Lemma 5.9}, the only difference in the quantum case being the presence of the $q$-factors in the products. 
For the coefficient of $\widehat{\vartheta}_{2v_{1}}$, we have 
 \[c(\gamma_{1}) = 1, \ c(\gamma_{2}) = 1, \ s(\gamma_{1}) = (1,0), \ s(\gamma_{2})  =(1,0),  \ \det( (1,0), (1,0) )  = 0\,,\]
 which implies $q^{\frac{1}{2}\det( (1,0), (1,0)  ) } = q^{0} = 1$. For the coefficient of $\widehat{\vartheta}_{0}=1$, we have 
\[c(\gamma_{1}) = 1, \ c(\gamma_{2}) = 1, \ s(\gamma_{1}) = (1,0), \ s(\gamma_{2})  =(-1,0),  \ \det( (1,0), (-1,0) )  = 0\,,\]
which implies $q^{\frac{1}{2}\det( (1,0), (-1,0)  ) } = q^{0} = 1$.  
	\end{proof}
    
\begin{lemma}
	\[\widehat{\vartheta}_{v_{1}}\widehat{\vartheta}_{v_{2}} = q^{\frac{1}{2}}\widehat{\vartheta}_{v_1 + v_2} + q^{-\frac{1}{2}} t^{D_3} \widehat{\vartheta}_{v_3}\,. \]
	\label{Lemma 7.0}
\end{lemma}
\begin{proof}
	Following the proof of Lemma \ref{Lemma 5.10}, we determine the relevant $q$-factors. For the coefficient of $\widehat{\vartheta}_{v_1+v_2}$, we have 
	\[c(\gamma_{1}) = 1, \ c(\gamma_{2}) = 1, \ s(\gamma_{1}) = (1,0), \ s(\gamma_{2})  =(0,1),  \ \det( (1,0), (0,1) )  = 1\,,\]
	 which implies $q^{\frac{1}{2}\det( (1,0), (0,1)  ) } = q^{\frac{1}{2}}$. For the coefficient of $\widehat{\vartheta}_{v_3}$, we have
		\[c(\gamma_{1}) = 1, \ c(\gamma_{2}) = 1, \ s(\gamma_{1}) = (-1,0), \ s(\gamma_{2})  =(0,1),  \ \det( (-1,0), (0,1) )  = -1\,,\] 
	which implies $q^{\frac{1}{2}\det( (-1,0), (0,1)  ) } = q^{-\frac{1}{2}}$. 
\end{proof}

\begin{lemma}
	$$\widehat{\vartheta}_{v_1 + v_2} \widehat{\vartheta}_{v_3}  = 
    q^{-1} t^{D_1} \widehat{\vartheta}_{2v_1} +  
    q t^{D_2} \widehat{\vartheta}_{2v_2} + t^{H} + t^{2D-H} + (q + q^{-1})t^D \,.$$
	\label{Lemma 7.1}
\end{lemma}
\begin{proof}
Following the proof of Lemma \ref{Lemma 5.11} and using Lemma \ref{Lemma_ray_quantum} for the function attached to the quantum ray $\rho_1$, we determine the relevant $q$-factors. 
	For the coefficient of  $\widehat{\vartheta}_{2v_{1}}$, we have 
	 	\[c(\gamma_{1}) = 1, \ c(\gamma_{2}) = 1, \ s(\gamma_{1}) = (1,1), \ s(\gamma_{2})  =(1,-1),  \ \det( (1,1), (1,-1) )  = -2\]
	 		 which implies $q^{\frac{1}{2}\det( (1,1), (1,-1)  ) } = q^{-1}$.
	 	For the coefficient of $\widehat{\vartheta}_{2v_{2}}$, we have 
	 \[c(\gamma_{1}) = 1, \ c(\gamma_{2}) = 1, \ s(\gamma_{1}) = (1,1), \ s(\gamma_{2})  =(-1,1),  \ \det( (1,1), (-1,1) )  = 2 \,,\]
	 which implies $q^{\frac{1}{2}\det( (1,1), (-1,1)  ) } = q^{1}$.
	 	 	For the coefficient of $\widehat{\vartheta}_{0}=1$, we have 
	 \begin{align*}c(\gamma_{1}) = 1, \ c(\gamma_{2}) = & t^H + t^{2D-H} + (q + q^{-1})t^D, \  s(\gamma_{1}) = (1,1), \ s(\gamma_{2})  =(-1,-1),  \ \det( (1,1), (-1,-1) )  = 0 \,,
	 \end{align*}
	 which implies $q^{\frac{1}{2}\det( (1,1), (-1,-1)  ) } = q^{0} = 1.$
\end{proof}

Finally, the following Theorem \ref{thm_main_2}, stated as Theorem \ref{thm_main_intro_2} in the introduction, determines explicitly the quantum intrinsic mirror algebra of an orbifold cubic surface with three nodes.

\begin{thm} \label{thm_main_2}
	Let $(Y,D)$ be an orbifold cubic surface over $\kk$, where $Y$ is a projective cubic surface with three nodes and $D \subset Y$ is a triangle of lines avoiding the nodes. 
The quantum intrinsic mirror algebra
$\widehat{R}_{(Y,D)}$
of $(Y,D)$ 
is the associative $\kk_q[\operatorname{NE}(Y)]$-algebra generated by three elements 
$\widehat{\vartheta}_{v_1}, \widehat{\vartheta}_{v_2}, \widehat{\vartheta}_{v_3}$ that satisfy the relations 
	\begin{align*}
		q^{-\frac{1}{2}}\widehat{\vartheta}_{v_{1}}\widehat{\vartheta}_{v_{2}} - q^{\frac{1}{2}}\widehat{\vartheta}_{v_{2}}\widehat{\vartheta}_{v_{1}} & = \left(q^{-1} - q \right) \widehat{\vartheta}_{v_{3}}\,, \\
		q^{-\frac{1}{2}}\widehat{\vartheta}_{v_{2}}\widehat{\vartheta}_{v_{3}} - q^{\frac{1}{2}}\widehat{\vartheta}_{v_{3}}\widehat{\vartheta}_{v_{2}} & = \left(q^{-1} - q \right) \widehat{\vartheta}_{v_{1}}\,, \\
		q^{-\frac{1}{2}}\widehat{\vartheta}_{v_{3}}\widehat{\vartheta}_{v_{1}} - q^{\frac{1}{2}}\widehat{\vartheta}_{v_{1}}\widehat{\vartheta}_{v_{3}} & = \left(q^{-1} - q \right) \widehat{\vartheta}_{v_{2}}\,, \\
	\end{align*}
and
	\begin{align*}q^{-\frac{1}{2}}\widehat{\vartheta}_{v_1}\widehat{\vartheta}_{v_2}
    \widehat{\vartheta}_{v_3} = 
    q^{-1} t^{D_1} \widehat{\vartheta}_{v_1}^2 + q t^{D_2}\widehat{\vartheta}_{v_2}^2 + 
    q^{-1} t^{D_3} \widehat{\vartheta}_{v_3}^2 & + t^{H} + t^{2 D-H} - (q+q^{-1}) t^{D} \,.\end{align*}
\end{thm}
\begin{proof}
Lemma \ref{Lemma 7.0} together with the $\Bbb{Z}/3\Bbb{Z}$ symmetry implies the first three commutator relations. 
For the cubic relation, using Lemma \ref{Lemma 7.0}, we have
\begin{align*}
	\widehat{\vartheta}_{v_1}\widehat{\vartheta}_{v_2}\widehat{\vartheta}_{v_3} & = \left( q^{\frac{1}{2}}\widehat{\vartheta}_{v_1 + v_2} + q^{-\frac{1}{2}} t^{D_3} \widehat{\vartheta}_{v_3} \right)\widehat{\vartheta}_{v_3}  = q^{\frac{1}{2}}\widehat{\vartheta}_{v_1 + v_2}\widehat{\vartheta}_{v_3} +  q^{-\frac{1}{2}} t^{D_3} \widehat{\vartheta}_{v_3}^2 \,,
\end{align*}
and then using Lemma \ref{Lemma 7.1}, we obtain
    \begin{align*}
\widehat{\vartheta}_{v_1}\widehat{\vartheta}_{v_2}\widehat{\vartheta}_{v_3}	& = q^{\frac{1}{2}} \left( q^{-1}t^{D_1}\widehat{\vartheta}_{2v_1} +  q t^{D_2} \widehat{\vartheta}_{2v_2} + t^H + t^{2D-H} + (q + q^{-1})t^D\right) + q^{-\frac{1}{2}} t^{D_3}\widehat{\vartheta}_{v_3}^2 \,. 
\end{align*}
Finally, by Lemma \ref{Lemma 6.5}, we have $\widehat{\vartheta}_{2v_1}=\widehat{\vartheta}_{v_1}^2-2 t^{D_2+D_3}$  and $\widehat{\vartheta}_{2v_2}=\widehat{\vartheta}_{v_2}^2-2t^{D_1+D_3}$. Therefore, using $D=D_1+D_2+D_3$, we obtain
\begin{align*} q^{-\frac{1}{2}}\widehat{\vartheta}_{v_1}\widehat{\vartheta}_{v_2}\widehat{\vartheta}_{v_3}	&=  
q^{-1} t^{D_1}(\widehat{\vartheta}_{v_1}^2-2 t^{D_2+D_3}) +  q t^{D_2}(\widehat{\vartheta}_{v_2}^2-2t^{D_1+D_3}) + t^H + t^{2D-H} + (q + q^{-1})t^D+ q^{-1} t^{D_3}\widehat{\vartheta}_{v_3}^2  \\
& = q^{-1} t^{D_1}\widehat{\vartheta}_{v_1}^2 + q t^{D_2}\widehat{\vartheta}_{v_2}^2 
+ q^{-1} t^{D_3}\widehat{\vartheta}_{v_3}^2 + t^{H} + t^{2 D-H} -(q+q^{-1})t^{D} \,, 
\end{align*}
and this concludes the proof.
\end{proof}

\bibliographystyle{plain}
\bibliography{bibliography}

\end{document}